\documentclass[12pt,english]{article}

\usepackage[T1]{fontenc}
\usepackage[utf8]{inputenc}
\usepackage[main=english]{babel}

\usepackage{geometry}
\usepackage{amsmath,amsthm,amssymb,amsfonts,mathtools}
\usepackage{mathrsfs}
\usepackage{esint}

\usepackage{graphicx}   
\usepackage{float}
\usepackage{tikz}       

\usepackage{comment}
\usepackage{indentfirst}
\usepackage{url}

\theoremstyle{definition}
\newtheorem{theorem}{Theorem}[section]
\newtheorem{lemma}[theorem]{Lemma}
\newtheorem{remark}[theorem]{Remark}
\newtheorem{corollary}[theorem]{Corollary}

\newtheorem{definition}[theorem]{Definition}
\newtheorem{proposition}[theorem]{Proposition}
\newtheorem{conjecture}[theorem]{Conjecture}
\title{Equivalence of Almgren--Pitts and phase-transition half-volume spectra}
\author{Talant Talipov}
\date{}

\begin{document}
\thispagestyle{empty}
\maketitle

\begin{abstract}
We prove that the Almgren--Pitts and phase-transition half-volume spectra of a closed Riemannian manifold are equal. This confirms a conjecture of Liam Mazurowski and Xin Zhou.
\end{abstract}

\section{Introduction}
\label{sec:intro}

The spectrum of the Laplacian is a fundamental invariant of a closed Riemannian manifold $(M^{n+1},g)$. A number $\lambda$ is called an eigenvalue of the Laplacian if there exists a nonzero function $u:M\to\mathbb{R}$ such that
\[
\Delta u + \lambda u = 0.
\]
It is well known that the eigenvalues form a discrete sequence
\[
0=\lambda_0<\lambda_1\le \lambda_2\le \cdots,
\qquad \lambda_p\to\infty \text{ as } p\to\infty.
\]
Moreover, the eigenvalues admit the min-max characterization
\[
\lambda_p=\inf_{(p+1)\text{-planes }P\subset W^{1,2}(M)}
\left[\sup_{u\in P\setminus\{0\}}
\frac{\int_M |\nabla u|^2}{\int_M u^2}\right],
\]
and satisfy the Weyl law
\[
\lambda_p\sim 4\pi^2 \operatorname{Vol}(B)^{-\frac{2}{n+1}}
\operatorname{Vol}(M)^{-\frac{2}{n+1}} p^{\frac{2}{n+1}}
\]
as $p\to\infty$, where $B$ denotes the unit ball in $\mathbb{R}^{n+1}$.

In \cite{G}, Gromov proposed a nonlinear analogue of the spectrum of the Laplacian. Roughly speaking, a $p$-sweepout of $M$ is a family $X$ of hypersurfaces with the property that, given any $p$ points in $M$, one can find a hypersurface $\Sigma\in X$ passing through all of them. He then defined the $p$-widths by
\[
\omega_p=\inf_{p\text{-sweepouts }X}
\left[\sup_{\Sigma\in X}\operatorname{Area}(\Sigma)\right].
\]
See Section \ref{sec:prelim} for the precise definition. The sequence $\{\omega_p\}_{p\in\mathbb N}$ is called the volume spectrum of $M$.

Gromov \cite{G1} and Guth \cite{Gu} proved that the volume spectrum has sublinear growth. More precisely, there exist constants $C_1,C_2>0$, depending on $M$, such that
\[
C_1 p^{\frac{1}{n+1}}\le \omega_p\le C_2 p^{\frac{1}{n+1}}.
\]
Later, Liokumovich, Marques, and Neves \cite{LMN} proved the Weyl law for the volume spectrum: there exists a universal constant $a_n$, depending only on the dimension, such that
\[
\omega_p\sim a_n \operatorname{Vol}(M)^{\frac{n}{n+1}}p^{\frac{1}{n+1}}
\]
as $p\to\infty$; see Chodosh--Mantoulidis \cite{CM2} for the computation of $a_n$ when $n=2$. The Weyl law for the volume spectrum has played a central role in the proof of several major results on the existence of minimal hypersurfaces.

In the early 1980s, Almgren \cite{A}, Pitts \cite{P}, and Schoen--Simon \cite{SS} developed a min-max theory for the area functional on closed Riemannian manifolds. Their combined work implies that every closed Riemannian manifold of dimension $3\le n+1\le 7$ contains a closed, smooth, embedded minimal hypersurface. Around the same time, Yau \cite{Y} conjectured that every closed manifold should contain infinitely many minimal hypersurfaces. Marques and Neves proposed a program to approach Yau's conjecture by using Almgren--Pitts min-max theory to produce a minimal hypersurface of area $\omega_p$ for each $p\in\mathbb N$.

This program has now been carried out. Let $(M^{n+1},g)$ be a closed Riemannian manifold with $3\le n+1\le 7$. Irie, Marques, and Neves \cite{IMN} proved that, for a generic metric $g$, the union of all minimal hypersurfaces in $M$ is dense. In particular, this established Yau's conjecture for generic metrics. Later, Marques, Neves, and Song \cite{MNS} strengthened this result by proving that, for a generic metric $g$, there exists a sequence of minimal hypersurfaces that becomes equidistributed in $M$. The Weyl law for the volume spectrum is a key ingredient in both arguments. Following Zhou's proof of the Multiplicity One Conjecture \cite{Z}, Marques and Neves \cite{MN2} showed that, for a generic metric $g$, there exists a sequence of minimal hypersurfaces $\Sigma_p$ with index $p$ and
\[
\operatorname{Area}(\Sigma_p)=\omega_p.
\]
Finally, Song \cite{Song} proved Yau's conjecture for arbitrary metrics $g$; see also \cite{Li} for the higher-dimensional analogue.

There is also a parallel min-max theory for minimal hypersurfaces based on phase transitions. This theory is built on the Allen--Cahn equation together with the varifold regularity theory of Hutchinson--Tonegawa--Wickramasekera \cite{HT,TW}. Gaspar and Guaraco \cite{GG2} defined a phase-transition spectrum $\{c(p)\}_{p\in\mathbb N}$ associated to a Riemannian manifold via the Allen--Cahn equation, and proved that each $c(p)$ is realized by a collection of minimal hypersurfaces with multiplicities. In ambient dimension three, Chodosh and Mantoulidis \cite{CM} proved the Multiplicity One Conjecture in the phase-transition setting. Thus, for generic metrics on $M^3$, they obtained a sequence of minimal hypersurfaces $\Sigma_p$ with index $p$ and
\[
\operatorname{Area}(\Sigma_p)=c(p).
\]

Gaspar and Guaraco \cite{GG} also proved a Weyl law for the phase-transition spectrum. Dey \cite{Dey} later proved that
\[
\omega_p=c(p)
\]
for all $p\in\mathbb N$, so the Almgren--Pitts volume spectrum and the phase-transition spectrum coincide.

In \cite{ZhouMazur0}, Mazurowski and Zhou introduced a half-volume analogue of this theory. In the Almgren--Pitts setting, one restricts to $p$-sweepouts by hypersurfaces that each enclose half the volume of $M$, and defines
\[
\tilde \omega_p=
\inf_{\text{half-volume $p$-sweepouts }X}
\left[\sup_{\Sigma\in X}\operatorname{Area}(\Sigma)\right].
\]
The sequence $\{\tilde\omega_p\}_{p\in\mathbb N}$ is called the Almgren--Pitts half-volume spectrum of $M$. Mazurowski and Zhou also defined an analogous phase-transition half-volume spectrum $\tilde c(p)$ by considering critical points of the Allen--Cahn energy subject to the constraint $\int_M u = 0$;
see Section \ref{sec:prelim} for the precise definition. In both settings they proved that the corresponding Weyl law continues to hold.

The main result of this paper is that these two half-volume spectra are equal.

\begin{theorem}
Let $\{\tilde\omega_p\}_{p=1}^\infty$ be the Almgren--Pitts half-volume spectrum of $M$, and let $\{\tilde c(p)\}_{p=1}^\infty$ be the phase-transition half-volume spectrum of $M$. Then, for every $p\in\mathbb N$,
\[
\tilde\omega_p=\tilde c(p).
\]
\end{theorem}

This confirms a conjecture of Mazurowski and Zhou \cite{ZhouMazur0}.

The main motivation for studying the half-volume spectrum comes from the following conjectures.

A well-known conjecture of Arnold \cite{arnold2004arnold} asserts that every Riemannian $2$-sphere $(S^2,g)$ admits at least two distinct closed curves with constant geodesic curvature $\kappa$ for every $\kappa>0$; see \cite{AB, CZ, KL, RS, Schneider, SchneiderAlex, ZZ-net} for partial results. The following \emph{Twin Bubble Conjecture} (see Zhou \cite{Zhou-ICM22}) may be viewed as a higher-dimensional analogue of Arnold's conjecture; see \cite{CZ-CMC, Dey23,GMK, Mazurowski22, ZhouMazur3, PX, Torralbo, Ye, ZZ} for partial results.

\begin{conjecture}[Twin Bubble Conjecture; X. Zhou \cite{Zhou-ICM22}]
\label{conj:two-cmcs}
Every closed Riemannian manifold $(M^{n+1},g)$ with $3\le n+1\le 7$ admits at least two distinct closed hypersurfaces with constant mean curvature $c$ for every $c>0$.
\end{conjecture}

The following conjecture may be regarded as a CMC analogue of Yau's conjecture.

\begin{conjecture}
\label{conj:CMC-Yau}
Every closed Riemannian manifold $(M^{n+1},g)$ with $3\le n+1\le 7$ admits infinitely many closed hypersurfaces with constant mean curvature enclosing half the volume of $M$.
\end{conjecture}

In \cite{ZhouMazur0}, Mazurowski and Zhou used the Allen--Cahn min-max theory together with the work of Bellettini--Wickramasekera \cite{bellettini2020inhomogeneous} to show that each value of the Allen--Cahn half-volume spectrum is realized by a constant mean curvature hypersurface enclosing half the volume of $M$, together with a collection of minimal hypersurfaces with even multiplicities.

In \cite{ZhouMazur1}, Mazurowski and Zhou proved a multiplicity one type result for the Almgren--Pitts half-volume spectrum. More precisely, they showed that for a generic metric on a closed manifold $M^{n+1}$ of dimension $3\le n+1\le 5$, each value $\tilde\omega_p$ is realized by a multiplicity one constant mean curvature hypersurface enclosing half the volume of $M$, thereby confirming Conjecture \ref{conj:CMC-Yau} in this case.

In \cite{ZZ}, Zhou and Zhu developed an Almgren--Pitts min-max theory for hypersurfaces of prescribed constant mean curvature $c$. However, that theory does not control the enclosed volume. Likewise, Bellettini and Wickramasekera \cite{bellettini2020inhomogeneous} developed an Allen--Cahn min-max theory for hypersurfaces of prescribed constant mean curvature $c$, again without control of the enclosed volume. Thus the two theories address a different constraint. The half-volume spectrum produces constant mean curvature hypersurfaces enclosing half the volume of $M$, but does not prescribe the value of the mean curvature.

In \cite{ZhouMazur2}, Mazurowski and Zhou showed that, for a generic metric on a closed manifold $(M^{n+1},g)$ with $3\le n+1\le 7$, Conjecture \ref{conj:two-cmcs} and Conjecture \ref{conj:CMC-Yau} are linked by the following alternative: for every $c>0$, either there exist infinitely many distinct smooth, closed, almost-embedded hypersurfaces in $M$ with constant mean curvature equal to $c$, or there exist infinitely many distinct smooth, closed, almost-embedded hypersurfaces in $M$ with constant mean curvature less than $c$ and enclosing half the volume of $M$.

\paragraph*{Outline.} In Section \ref{sec:prelim} we introduce the notation and recall the definitions of the half-volume spectra. In Section \ref{sec:phase-to-cycles} we prove
\[
\tilde\omega_p(M)\le \frac{1}{2\sigma}\liminf_{\varepsilon\to 0^+}\tilde c_\varepsilon(p).
\]
In Section \ref{sec:cycles-to-phase} we prove the reverse inequality.

In Section \ref{sec:phase-to-cycles} we follow the ideas of Gaspar and Guaraco \cite{GG2}. Starting from a mean-zero family of $H^1(M)$-functions, we first construct a discrete family of cycles from their level sets, and then extend it to a mass-continuous family. The main difficulty is that the half-volume condition is not automatically preserved in this procedure. To address this, we first show that the mean-zero assumption implies that the resulting family is close to half-volume. We then construct a half-volume correction map, which projects such families to exactly half-volume cycles.

In the reverse direction, in Section \ref{sec:cycles-to-phase}, we follow the ideas of Dey \cite{Dey}. Starting from a flat-continuous family of cycles, we first approximate it by a discrete family of smooth hypersurfaces. We then refine this to a family of hypersurfaces that are smooth away from a fixed codimension-two set and satisfy suitable locality properties. Using the signed distance functions to these hypersurfaces, we construct a discrete family of Lipschitz functions. Next, we extend this family from the $0$-skeleton to the whole cubical complex by using Dey's gluing construction. Here the main difficulty is that the resulting family need not be mean-zero. However, the locality properties of the almost-smooth hypersurfaces allow us to control the integrals of the glued functions. This makes it possible to construct a mean-zero correction map and thereby modify the final family so that it satisfies the mean-zero constraint.

\paragraph*{Acknowledgments.} The author is grateful to Yevgeny Liokumovich for his invaluable supervision and encouragement throughout this work. This work was supported by the Dr. Sergiy and Tetyana Kryvoruchko Graduate Scholarship in Mathematics.

\section{Preliminaries}
\label{sec:prelim}

Let $(M^{n+1},g)$ be a closed Riemannian manifold. We adopt the following notation.
\begin{itemize}
\item Let $\mathfrak h = \frac{1}{2}\operatorname{Vol} M$.
\item Let $\mathcal C(M)$ denote the collection of Caccioppoli sets in $M$.
\item Let $\mathcal C_{\mathfrak h}(M)$, $\mathcal C_{\ge \mathfrak h}(M)$, and $\mathcal C_{\le \mathfrak h}(M)$ denote the spaces of Caccioppoli sets with volume equal to $\mathfrak h$, at least $\mathfrak h$, and at most $\mathfrak h$, respectively.
\item Let $\mathcal Z(M,\mathbb{Z}_2)$ denote the space of all $n$-dimensional flat chains mod 2 in $M$.
\item Let $\mathcal B(M,\mathbb{Z}_2)$ denote the set of all $T\in \mathcal Z(M,\mathbb{Z}_2)$ for which $T = \partial \Omega$ for some $\Omega\in \mathcal C(M)$. This is the connected component of the empty set in $\mathcal Z(M,\mathbb{Z}_2)$ with respect to the flat topology.
\item Let $\mathcal H(M,\mathbb{Z}_2)$ be the set of all $T \in \mathcal Z(M,\mathbb{Z}_2)$ such that $T = \partial \Omega$ for some $\Omega\in \mathcal{C}_{\mathfrak h}(M)$. We refer to this as the space of ``half-volume cycles.''
\item We write $\mathcal F$ for the flat topology, $\mathbf F$ for the $\mathbf F$-topology, and $\mathbf M$ for the mass topology. Unless stated otherwise, all spaces are equipped with the flat topology.
\item By abuse of notation, we write $\operatorname{Vol}(\Omega)$ and $\operatorname{Area}(T)$ in place of $\mathbf M(\Omega)$ and $\mathbf M(T)$ for $\Omega\in \mathcal C(M)$ and $T\in \mathcal Z(M,\mathbb{Z}_2)$, respectively.
\end{itemize}

\subsection{Half-volume spectrum}
\label{subsec:half-volume}

In this section, we briefly recall the definitions of the half-volume spectrum and the phase-transition half-volume spectrum introduced by Mazurowski and Zhou \cite{ZhouMazur0}.

Let $X$ be a cubical complex. Since $\mathcal B(M,\mathbb{Z}_2)$ is weakly homotopy equivalent to $\mathbb{R}\text{P}^\infty$ \cite{MN2}, its cohomology ring with $\mathbb{Z}_2$ coefficients is $\mathbb{Z}_2[\bar\lambda]$, where the generator $\bar\lambda$ has degree 1.

\begin{definition}
\label{ps}
A flat continuous map $\Phi: X\to \mathcal B(M,\mathbb{Z}_2)$ is called a $p$-sweepout if $\Phi^*\bar\lambda^p \neq 0$ in $H^p(X,\mathbb{Z}_2)$.
\end{definition}

\begin{definition}
A map $\Phi: X\to \mathcal Z(M,\mathbb{Z}_2)$ is said to have {\it no concentration of mass} if
\[
\lim_{r\to 0} \left[\sup_{q\in M} \sup_{x\in X} \operatorname{Area}(\Phi(x) \llcorner B(q,r))\right] = 0.
\]
\end{definition}

\begin{definition}
Let $\mathcal P_p(M)$ denote the collection of all $p$-sweepouts of $M$ with no concentration of mass.
\end{definition}

\begin{definition}
The $p$-width of $M$ is
\[
\omega_p = \inf_{\Phi\in \mathcal P_p(M)} \left[\sup_{x\in \operatorname{dom}(\Phi)} \operatorname{Area}(\Phi(x))\right].
\]
\end{definition}

By Corollary 8 in \cite{ZhouMazur0}, the space \(\mathcal H(M,\mathbb{Z}_2)\) is homotopy equivalent to \(\mathcal B(M,\mathbb{Z}_2)\). In particular, \(H^*(\mathcal H(M,\mathbb{Z}_2);\mathbb{Z}_2)\cong \mathbb{Z}_2[\lambda]\), where \(\lambda\) is the degree-one generator.

\begin{definition}
A flat continuous map $\Phi: X\to \mathcal H(M,\mathbb{Z}_2)$ is called a half-volume $p$-sweepout if $\Phi^* \lambda^p \neq 0$ in $H^p(X,\mathbb{Z}_2)$.
\end{definition}

\begin{definition}
Let $\mathcal Q_p(M)$ denote the collection of all half-volume $p$-sweepouts of $M$ with no concentration of mass.
\end{definition}

\begin{definition}
The half-volume $p$-width of $M$ is
\[
\tilde \omega_p = \inf_{\Phi\in \mathcal Q_p(M)}\left[\sup_{x\in \operatorname{dom}(\Phi)} \operatorname{Area}(\Phi(x))\right].
\]
\end{definition}

We call the sequence $\{\tilde \omega_p\}_{p\in \mathbb{N}}$ the half-volume spectrum of $M$.

We may assume that the domains appearing in the min-max sequence are finite cubical complexes of dimension at most $p$; see Lemma 2.25 in \cite{Staffa}.

There is a corresponding half-volume spectrum in the Allen--Cahn setting. Let $W: \mathbb{R}\to \mathbb{R}$ be an even double-well potential. Thus:
\begin{itemize}
\item[(i)] $W$ is smooth and non-negative,
\item[(ii)] $W(x) = W(-x)$ for all $x\in \mathbb{R}$,
\item[(iii)] $W$ has non-degenerate minima $W(\pm1)=0$, \label{property(iii)}
\item[(iv)] $W$ has a non-degenerate maximum $W(0)>0$,
\item[(v)] $W$ is increasing on $(-1,0)$ and $(1,\infty)$ and decreasing on $(0,1)$ and $(-\infty,-1)$,
\item[(vi)] there are constants $\kappa > 0$ and $\alpha\in(0,1)$ such that $W''(x) \ge \kappa$ for all $\vert x\vert \ge \alpha$,
\item[(vii)] there are constants $0 < C_1 < C_2$, $\beta > 1$, and $2 < q < \frac{11}{5}$ such that
\[
C_1 \vert x\vert^q \le W(x) \le C_2 \vert x\vert^q
\quad \text{and} \quad
C_1 \vert x\vert^{q-1} \le \vert W'(x)\vert \le C_2 \vert x\vert^{q-1}
\]
for all $\vert x\vert \ge \beta$.
\end{itemize}

Define
\[
\sigma = \int_{-1}^1 \sqrt{W(s)/2}\, ds.
\]
For a function $u: M\to \mathbb{R}$ in $W^{1,2}$ and $\varepsilon > 0$, define the Allen--Cahn energy by
\[
E_\varepsilon(u) = \int_M \frac \varepsilon 2 \vert  \nabla u\vert^2 + \frac{W(u)}{\varepsilon}.
\]

To define the spectrum, we need some additional background. A paracompact topological space $X$ is called a $\mathbb{Z}_2$-space if it carries a free $\mathbb{Z}_2$-action. For such a space, there is a quotient $T = X/\mathbb{Z}_2$, and the natural projection $X\to T$ is a principal $\mathbb{Z}_2$-bundle. Every such bundle arises as the pullback of the universal bundle $S^\infty \to \mathbb{R}\text{P}^\infty$. More precisely, there is a classifying map $f: T\to \mathbb{R}\text{P}^\infty$ such that $X\to T$ is the pullback of $S^\infty \to \mathbb{R}\text{P}^\infty$ via $f$. The Alexander--Spanier cohomology ring of $\mathbb{R}\text{P}^\infty$ with $\mathbb{Z}_2$ coefficients is $\mathbb{Z}_2[\mu]$, where $\mu$ is the degree-one generator. The map $f$ is unique up to homotopy, and therefore the classes $f^*\mu^p$ are well-defined in the Alexander--Spanier cohomology ring $H^*(T,\mathbb{Z}_2)$. The $\mathbb{Z}_2$-index of $X$ is the largest $p$ such that $f^*\mu^{p-1} \neq 0$ in $H^*(T,\mathbb{Z}_2)$. A subset $A\subset X$ is called invariant if it is preserved by the $\mathbb{Z}_2$-action.

The $\mathbb{Z}_2$-index has the following basic properties (see \cite{FR}).
\begin{itemize}
\item[(i)] (Monotonicity) If $X_1$ and $X_2$ are $\mathbb{Z}_2$-spaces and there is a continuous equivariant map $X_1\to X_2$, then $\operatorname{ind}_{\mathbb{Z}_2}(X_1) \le \operatorname{ind}_{\mathbb{Z}_2}(X_2)$.

\item[(ii)] (Subadditivity) If $X$ is a $\mathbb{Z}_2$-space and $A_1$, $A_2$ are closed invariant subsets with $A_1\cup A_2 = X$, then
\[
\operatorname{ind}_{\mathbb{Z}_2}(X) \le \operatorname{ind}_{\mathbb{Z}_2}(A_1) + \operatorname{ind}_{\mathbb{Z}_2}(A_2).
\]

\item[(iii)] (Continuity) If $X$ is a $\mathbb{Z}_2$-space and $A$ is a closed invariant subset of $X$, then there exists an invariant neighborhood $V$ of $A$ in $X$ such that
\[
\operatorname{ind}_{\mathbb{Z}_2}(A) = \operatorname{ind}_{\mathbb{Z}_2}(\overline V).
\]
\end{itemize}

The space $W^{1,2}(M)\setminus \{0\}$ is paracompact since it is a metric space. It also has the natural $\mathbb{Z}_2$-action given by $u\mapsto -u$. Since $W$ is even, the functional $E_\varepsilon$ is invariant under this action, that is, $E_\varepsilon(u) = E_\varepsilon(-u)$. A subset $A\subset W^{1,2}(M)\setminus\{0\}$ is called invariant if $u\in A$ exactly when $-u\in A$. The $\mathbb{Z}_2$-action then restricts to any such $A$. Define
\[
\mathcal F_p = \{A\subset W^{1,2}(M)\setminus\{0\}:\ A\text{ is compact and invariant with } \operatorname{ind}_{\mathbb{Z}_2}(A) \ge p+1\}.
\]
Gaspar and Guaraco \cite{GG2} define the min-max values
\[
c_\varepsilon(p) = \inf_{A\in \mathcal F_p} \left[\sup_{u\in A} E_\varepsilon(u)\right].
\]
They then set $c(p) = \liminf_{\varepsilon\to 0} c_\varepsilon(p)$. The sequence $\{c(p)\}_{p\in \mathbb{N}}$ is called the phase transition spectrum of $M$.

Now define
\[
Y = \{u\in W^{1,2}(M): \int_M u = 0\}.
\]
Since $Y$ is a closed subspace of $W^{1,2}(M)$, it is again a Hilbert space. We can carry out the same construction with $Y$ in place of $W^{1,2}(M)$. For each $p\in \mathbb{N}$, define
\[
\mathcal G_p = \{A\subset Y\setminus\{0\}:\, A \text{ is compact and invariant with } \operatorname{ind}_{\mathbb{Z}_2}(A) \ge p+1\},
\]
and set
\[
\tilde c_\varepsilon(p) = \inf_{A\in \mathcal G_p}\left[\sup_{u\in A} E_\varepsilon(u)\right].
\]
Taking the limit as $\varepsilon \to 0$ yields the phase-transition half-volume spectrum.

\begin{definition}
For each $p\in \mathbb{N}$, let $\tilde c(p) = \liminf_{\varepsilon\to 0} \tilde c_\varepsilon(p)$. The phase-transition half-volume spectrum of $M$ is the sequence $\{\tilde c(p)\}_{p\in \mathbb{N}}$.
\end{definition}

\begin{remark}
Mazurowski and Zhou \cite{ZhouMazur0} include the normalization constant $1/2\sigma$ in the definition of $\tilde c_\varepsilon(p)$ and $\tilde c(p)$. We omit this factor in order to match the notation in \cite{Dey, GG2}.
\end{remark}
\subsection{Cell complexes and subdivisions}
\label{subsec:cell_complexes}

For $k\in\mathbb{N}$, let $I(k)$ denote the cell complex on $I=[0,1]$ whose $0$-cells are
\[
[0],\ [k^{-1}],\ [2k^{-1}],\ \dots,\ [1-k^{-1}],\ [1],
\]
and whose $1$-cells are the closed intervals
\[
[0,k^{-1}],\ [k^{-1},2k^{-1}],\ \dots,\ [1-k^{-1},1].
\]
For $m\in\mathbb{N}$, let $I^m(k)$ denote the product cell complex on $I^m=[0,1]^m$ whose cells are of the form
\[
\alpha_1\otimes \alpha_2\otimes \cdots \otimes \alpha_m,
\qquad \alpha_j\in I(k).
\]
We identify a cell $\alpha_1\otimes\cdots\otimes\alpha_m$ with its support
$\alpha_1\times\cdots\times\alpha_m\subset I^m$. Note that $I^m(1)$ is the standard cell complex on $I^m$.

For each $m\in\mathbb{N}$, define the involution
\[
\tau_m:I^m\to I^m,
\qquad
\tau_m(x_1,\dots,x_m):=(1-x_1,\dots,1-x_m).
\]
When the dimension is clear, we simply write $\tau$.

A cubical subcomplex $X\subset I^m(1)$ is called \emph{symmetric} if $\tau(X)=X$. Since $\tau$ sends cells of $I^m(k)$ to cells of $I^m(k)$, every subdivision $X(k)$ of a symmetric cubical complex is again symmetric.

If $X$ is symmetric, we regard $\tau$ as the free $\mathbb{Z}_2$-action on $X$.
A map $h:X\to Y$ is called \emph{odd} if $h(\tau x)=-h(x)$ for all $x\in X$, and a map $\Phi:X\to Z$ is called \emph{even} if $\Phi(\tau x)=\Phi(x)$ for all $x\in X$.

If $Y$ is a subcomplex of $I^m(1)$, define $Y(k)$ to be the union of all cells of $I^m(k)$
whose support is contained in some cell of $Y$.
If $Y$ is a cell complex, we write $Y_p$ for the collection of its $p$-cells.
If $\beta,\alpha$ are cells with $\beta$ a face of $\alpha$ (we allow $\dim\beta<\dim\alpha$,
not necessarily codimension one), we write $\beta\prec\alpha$, and we write $\beta\preceq\alpha$
if $\beta\prec\alpha$ or $\beta=\alpha$.

If $\lambda=[ik^{-1},(i+1)k^{-1}]\in I(k)$ is a $1$-cell, define the canonical map
\[
\Delta_\lambda:I\to I,\qquad \Delta_\lambda(t)=(i+t)k^{-1},
\]
so that $\Delta_\lambda:I\to\lambda$ is a homeomorphism.
More generally, if $\alpha=\alpha_1\otimes\cdots\otimes\alpha_m\in I^m(k)_p$, there exist precisely
$p$ indices $1\le j_1<\cdots<j_p\le m$ such that $\dim(\alpha_{j_s})=1$.
Define $\Delta_\alpha:I^p\to I^m$ by
\[
(\Delta_\alpha(t_1,\dots,t_p))_j=
\begin{cases}
\Delta_{\alpha_{j_s}}(t_s), & j=j_s \text{ for some }s,\\
\alpha_j, & \dim(\alpha_j)=0,
\end{cases}
\]
so that $\Delta_\alpha:I^p\to \alpha$ is a homeomorphism. Let $D_\alpha:\alpha\to I^p$ be its inverse.
If $\beta\prec\alpha$ and $D_\alpha(\beta)=\hat\beta$ is a face of $I^p(1)$, then the compatibility relation $\Delta_\alpha\circ\Delta_{\hat\beta}=\Delta_\beta$ holds.

A \emph{cubical complex} is a subcomplex $X\subset I^K(1)$ for some $K\in\mathbb{N}$.
For $k\in\mathbb{N}$, we define $X(k)$ to be the cell complex on $X$ induced by the subdivision $I^K(k)$,
i.e.\ the union of all cells of $I^K(k)$ whose support lies in a cell of $X$.

Now let $\pi:\widehat X\to \tilde X$ be a double cover, where $\tilde X$ is a cubical complex.
For $k\in\mathbb{N}$, define $\widehat X(k)$ to be the cell complex whose cells are the connected components
of $\pi^{-1}(c)$ as $c$ ranges over the cells of $\tilde X(k)$.
Since each cell $c$ is contractible, $\pi^{-1}(c)$ splits as a disjoint union of two cells
\[
\pi^{-1}(c)=\hat c_1\sqcup \hat c_2,
\]
and the restrictions $\pi:\hat c_i\to c$ are homeomorphisms.

\subsection{Lifting via the pullback double cover}
\label{subsec:lifting}

We will need a continuous choice of Caccioppoli representatives for a family
$\Phi:X\to \mathcal B(M,\mathbb{Z}_2)$. The correct framework is that the
boundary map from Caccioppoli sets is a two-sheeted covering.

Let $X$ be a $p$-dimensional symmetric cubical subcomplex, equipped with the involution $\tau$ from Subsection~\ref{subsec:cell_complexes}, and let $\tilde X:=X/\{x\sim \tau(x)\}$ be the orbit space. We assume that $\Phi:X\to \mathcal B(M,\mathbb{Z}_2)$ is even, i.e.\ $\Phi(\tau x)=\Phi(x)$, so it induces a flat-continuous map
\[
\tilde \Phi:\tilde X\to \mathcal B(M,\mathbb{Z}_2),
\qquad
\tilde \Phi([x])=\Phi(x).
\]

Equip $\mathcal C(M)$ with the $L^1$-metric
\[
d_{L^1}(\Omega_1,\Omega_2):=\operatorname{Vol}(\Omega_1\Delta\Omega_2).
\]
By \cite{MN2}, the boundary map
\[
\partial:\mathcal C(M)\longrightarrow \mathcal B(M,\mathbb{Z}_2),\qquad \Omega\longmapsto \partial[\![\Omega]\!],
\]
is a double cover. Its deck transformation is given by the complement map $\Omega\mapsto M\setminus\Omega$.

Define the pullback space
\[
\widehat X
:=\{(\tilde x,\Omega)\in \tilde X\times \mathcal C(M):\ \partial\Omega=\tilde \Phi(\tilde x)\},
\qquad
\pi(\tilde x,\Omega):=\tilde x.
\]
Since $\widehat X$ is the pullback of the covering $\partial:\mathcal C(M)\to\mathcal B(M,\mathbb{Z}_2)$ along $\tilde \Phi$, the map $\pi:\widehat X\to \tilde X$ is again a double cover. For each $\tilde x\in \tilde X$, the fiber $\pi^{-1}(\tilde x)$ consists of exactly two points $(\tilde x,\Omega)$ and $(\tilde x,M\setminus\Omega)$.

Let
\[
\iota:\widehat X\to\widehat X,
\qquad
\iota(\tilde x,\Omega):=(\tilde x,M\setminus\Omega).
\]
Then $\iota$ is the deck involution of the double cover $\pi:\widehat X\to\tilde X$.

Define the lifted family of sets by taking the second projection:
\[
\widehat\Omega:\widehat X\to \mathcal C(M),
\qquad
\widehat\Omega(\tilde x,\Omega):=\Omega.
\]
Then by construction $\partial\widehat\Omega=\tilde \Phi\circ \pi$ and $\widehat\Omega(\iota\hat x)=M\setminus \widehat\Omega(\hat x)$
for all $\hat x\in\widehat X$. Moreover, $\widehat\Omega$ is $L^1$-continuous: $\widehat X$ carries the subspace topology from $\tilde X\times\mathcal C(M)$ and $\widehat\Omega$ is the restriction of the continuous projection $(\tilde x,\Omega)\mapsto \Omega$.

\begin{definition}\label{def:dhv}
Let $T\in \mathcal B(M,\mathbb{Z}_2)$ be a boundary cycle. Choose any Caccioppoli set $\Omega\in \mathcal C(M)$ with $\partial\Omega=T$ and define the \emph{half-volume defect} of $T$ by
\[
d_{\mathfrak h}(T):=\bigl|\operatorname{Vol}(\Omega)-\mathfrak h\bigr|.
\]
\end{definition}

The boundary map $\partial:\mathcal C(M)\to \mathcal B(M,\mathbb{Z}_2)$ is a double cover with deck transformation $\Omega\mapsto M\setminus\Omega$. Hence, if $\Omega_1,\Omega_2\in\mathcal C(M)$ satisfy $\partial\Omega_1=\partial\Omega_2=T$, then $\Omega_2=\Omega_1$ or $\Omega_2=M\setminus\Omega_1$ (up to null sets). Since $\mathfrak h=\frac12\operatorname{Vol}(M)$,
\[
\bigl|\operatorname{Vol}(M\setminus\Omega_1)-\mathfrak h\bigr|=\bigl|\operatorname{Vol}(\Omega_1)-\mathfrak h\bigr|,
\]
so $d_{\mathfrak h}(T)$ does not depend on the choice of $\Omega$.

\section{From level-sets to cycles} \label{sec:phase-to-cycles}

In this section, we prove the following inequality.

\begin{theorem} \label{thm:wplp}
For every $p \in \mathbb{N}$, it holds
	\[\tilde\omega_p(M) \leq \frac{1}{2\sigma}\liminf_{\varepsilon \to 0^+} \tilde c_\varepsilon(p).\]
\end{theorem}

The $p$-widths, $\tilde \omega_p(M)$, are defined in terms of maps $\Phi: X \to \mathcal H(M,\mathbb{Z}_2)$ where $X$ is a cubical complex. However, the min-max values $\tilde c_\varepsilon(p)$ are defined in terms of the elements of $\mathcal{G}_p$ which can be very different from continuous images of cubical complexes. Because of this, in order to prove Theorem \ref{thm:wplp} we will first approximate a set $A \in \mathcal{G}_p$ which is almost optimal (in the sense that its energy is close to $\tilde c_\varepsilon(p)$) by the image of an odd map $h$ from a $p$-dimensional cubical complex into $Y$.

\subsection{Cubical subcomplexes and min-max values}

Initially, we need to show that the min-max value $\tilde c_\varepsilon(p)$ can be obtained by restricting ourselves to sets which are the image of certain $p$-dimensional subcomplexes of $I^m(k)$ by odd maps into $Y$.

Fix $p \in \mathbb{N}$ and denote by $\mathcal{C}_p$ the family of all $X$ that are $p$-dimensional symmetric cubical subcomplexes of $I^m(k)$, for some $m,k \in \mathbb{N}$, with $\operatorname{ind}_{\mathbb{Z}_2}(X)\geq p+1$. For every such $X$, we consider also the family $\tilde \Gamma(X)$ of all continuous odd maps $h:X \to Y$ and its associated min-max values
	\[\tilde c_\varepsilon(X) = \inf_{h \in \tilde \Gamma(X)} \sup_{x\in h(X)} E_\varepsilon(x).\]
By the monotonicity property of the index, we have $h(X) \in \mathcal{G}_p$ for all $h \in \tilde\Gamma(X)$, thus $\tilde c_\varepsilon(p)\leq \tilde c_\varepsilon(X)$. Moreover, we have (compare with Lemma 6.2 in \cite{GG2})

\begin{lemma} \label{lemma_pdim}
For all $p \in \mathbb{N}$, it holds
	\[\tilde c_\varepsilon(p) = \inf_{X \in \mathcal{C}_p} \tilde c_\varepsilon(X).\]
\end{lemma}

\begin{proof}
Given $\delta>0$, let $A_0 \in \mathcal{G}_p$ be such that $\sup E_\varepsilon(A_0) \leq \tilde c_\varepsilon(p) + \delta/2$. Given an arbitrary neighborhood $U$ of $A_0$ in $Y$, we can find a subspace $E\subset Y$ with $m:=\dim E<+\infty$ and $A \subset U \cap E$ such that $\operatorname{ind}_{\mathbb{Z}_2}(A)=\operatorname{ind}_{\mathbb{Z}_2}(A_0)$ (see Proposition 3.1 in \cite{Marzocchi}). Choose a linear isomorphism $L:\mathbb{R}^m\to E$ and define an affine map
\[
T:I^m\to E,
\qquad
T(x):=L(2x-\mathbf 1),
\]
where $\mathbf 1=(1,\dots,1)\in\mathbb{R}^m$.
After rescaling $L$, we may assume that $T(I^m)$ is a cube in $Y$
containing $A$ in its interior. Moreover, $T(\tau x)=-T(x)$ for all $x\in I^m$.

Set $A':=T^{-1}(A)\subset I^m$ and $U':=T^{-1}(U)\subset I^m$. Then $A'$ is compact and invariant, $U'$ is an invariant neighborhood of $A'$,
and
\[
\operatorname{ind}_{\mathbb{Z}_2}(A')=\operatorname{ind}_{\mathbb{Z}_2}(A)=\operatorname{ind}_{\mathbb{Z}_2}(A_0)\ge p+1
\]
by monotonicity and the fact that $T|_{A'}:A'\to A$ is an odd homeomorphism.

Choose $k\in\mathbb{N}$ such that $T(\alpha)\subset U$ for every $m$-cell $\alpha\in I^m(k)_m$ with $\alpha\cap A'\neq\emptyset$.
Let $X_m$ be the union of all such $m$-cells. Then $X_m$ is a symmetric cubical subcomplex of $I^m(k)$, and $A'\subset X_m\subset U'$. Hence, if $U$ was chosen so that $\operatorname{ind}_{\mathbb{Z}_2}(\overline U)=\operatorname{ind}_{\mathbb{Z}_2}(A)$, then also $\operatorname{ind}_{\mathbb{Z}_2}(X_m)\ge p+1$.

Let $X$ be the $p$-skeleton of $X_m$. Then $X$ is a $p$-dimensional symmetric
cubical subcomplex of $I^m(k)$. Consider the odd continuous map $h:=T|_X:X\to Y$. Since $T(X)\subset U$, we have
\[
\sup_{x\in X}E_\varepsilon(h(x))
\le \sup E_\varepsilon(U).
\]
Therefore
\[
\tilde c_\varepsilon(X)\le \sup E_\varepsilon(U).
\]

It remains to check that $X\in\mathcal C_p$. It suffices to observe that $H^p(X_m,X;\mathbb{Z}_2)=0$. The exact cohomology sequence of the pair $(X_m,X)$ then implies that the inclusion
$X\hookrightarrow X_m$ induces an injective map
\[
H^p(X_m;\mathbb{Z}_2)\to H^p(X;\mathbb{Z}_2),
\]
and hence
\[
\operatorname{ind}_{\mathbb{Z}_2}(X)\ge \operatorname{ind}_{\mathbb{Z}_2}(X_m)\ge p+1.
\]
Thus $X\in\mathcal C_p$.
\end{proof}

\subsection{Proof of Theorem \ref{thm:wplp}}

Theorem \ref{thm:wplp} is a direct consequence of the following approximation result which is proved later in this section:

\begin{theorem} \label{thm:comp}
Fix $\tilde \sigma \in (0,\sigma/2)$. There exist positive constants $C$ and $\delta_0$ with the following property. Given $\delta_1 \in (0,\delta_0)$, there exists a sequence $\varepsilon_\nu=\varepsilon_\nu(\delta_1)\downarrow 0$ such that
\[
\tilde c_{\varepsilon_\nu}(p)\to \tilde c(p)=\liminf_{\varepsilon\to 0^+}\tilde c_\varepsilon(p),
\]
and such that for every $\nu$ sufficiently large and every $X \in \mathcal{C}_p$ with
\[
\tilde c_{\varepsilon_\nu}(X) \leq \tilde c_{\varepsilon_\nu}(p)+\varepsilon_\nu,
\]
there exists an even map $\Phi:X \to \mathcal H(M, \mathbb{Z}_2)$ which is continuous with respect to the flat topology and satisfies
\[
\sup_{x \in X} {\bf M}(\Phi(x)) \leq \frac{\tilde c_{\varepsilon_\nu}(p)+2\varepsilon_\nu}{4\tilde \sigma} + C\delta_1.
\]
Moreover, the map $\tilde \Phi:\tilde X \to \mathcal H(M, \mathbb{Z}_2)$ induced by $\Phi$ in the orbit space $\tilde X=X/\{x\sim \tau(x)\}$, is a half-volume $p$-sweepout.
\end{theorem}

Let us show now how this result implies Theorem \ref{thm:wplp}.

\begin{proof}[Proof of Theorem \ref{thm:wplp}]
Fix $\tilde \sigma \in (0,\sigma/2)$ and $\delta_1\in(0,\delta_0)$, where $\delta_0$ is given by Theorem \ref{thm:comp}. By Theorem \ref{thm:comp}, there exists a sequence $\varepsilon_\nu\downarrow 0$ such that
\[
\tilde c_{\varepsilon_\nu}(p)\to \tilde c(p)=\liminf_{\varepsilon\to 0^+}\tilde c_\varepsilon(p).
\]
For each sufficiently large $\nu$, Lemma \ref{lemma_pdim} gives some $X_\nu\in\mathcal C_p$ such that
\[
\tilde c_{\varepsilon_\nu}(X_\nu)\le \tilde c_{\varepsilon_\nu}(p)+\varepsilon_\nu.
\]
Applying Theorem \ref{thm:comp} to $X_\nu$, we obtain
\[
\tilde \omega_p(M) \leq \frac{\tilde c_{\varepsilon_\nu}(p)+2\varepsilon_\nu}{4\tilde \sigma}+C \delta_1.
\]
Passing to the limit as $\nu\to\infty$, we obtain
\[
\tilde \omega_p(M) \leq \frac{\tilde c(p)}{4\tilde \sigma}+C\delta_1.
\]
Since $\delta_1>0$ is arbitrary, letting $\delta_1\downarrow 0$ gives
\[
\tilde \omega_p(M) \leq \frac{\tilde c(p)}{4\tilde \sigma}.
\]
Finally, as $\tilde \sigma \uparrow \sigma/2$,
\[
\tilde \omega_p(M) \leq \frac{1}{2\sigma}\tilde c(p)
= \frac{1}{2\sigma}\liminf_{\varepsilon \to 0^+} \tilde c_\varepsilon(p).
\]
\end{proof}

\subsection{Construction of a discrete map fine in the flat topology}

Gaspar and Guaraco described how to obtain discrete even maps into $\mathcal B(M,\mathbb{Z}_2)$ which are fine in the flat metric, from maps $h:X\to H^1(M)$ (see Section 6.7 in \cite{GG2}).

\begin{proposition}[Gaspar--Guaraco, \cite{GG2}]
\label{prop:GG_flat_fine_discrete}
Fix $\tilde\sigma\in(0,\sigma/2)$. Let $X^p$ be a symmetric cubical complex, equipped with the involution $\tau$ from Subsection~\ref{subsec:cell_complexes}, and let $h:X\to H^1(M)$ be continuous and odd. Denote $h_x := h(x)$ for $x \in X$. Fix $\varepsilon>0$ and set
\[
L_\varepsilon := \sup_{x\in X} E_\varepsilon(h(x))<\infty.
\]
Then there exist:
\begin{enumerate}
\item a number $\delta=\delta(\tilde\sigma)\in(0,1)$,
\item a subdivision level $j\in\mathbb{N}$,
\item and numbers $s_x\in(-1+\delta,\,1-\delta)$ for each vertex $x\in X(j)_0$
      satisfying $s_{\tau x}=-s_x$ and $\operatorname{Vol}(\{h_x=s_x\})=0$,
\end{enumerate}
such that the discrete map
\[
\phi_0:X(j)_0\to \mathcal B(M,\mathbb{Z}_2),
\qquad
\phi_0(x):=\partial\big[\![\{h_x>s_x\}]\!\big],
\]
has the following properties:
\begin{enumerate}
\item[(i)] \textbf{Evenness:} $\phi_0(\tau x)=\phi_0(x)$ for all $x\in X(j)_0$.
Moreover, setting $\Omega(x):=[\![\{h_x>s_x\}]\!]$ one has
$\Omega(x)+\Omega(\tau x)=[\![M]\!]$ in $\mathbf I_{n+1}(M;\mathbb{Z}_2)$.
\item[(ii)] \textbf{Mass bound:} for every $x\in X(j)_0$,
\[
\mathbf M(\phi_0(x)) \le \frac{E_\varepsilon(h_x)}{4\tilde\sigma}\le \frac{L_\varepsilon}{4\tilde\sigma}.
\]
\item[(iii)] \textbf{Flat-fineness:} there exists a constant $C=C(\tilde\sigma,W)>0$
such that for any pair $x,y\in X(j)_0$ lying in a common $p$-cell,
\[
\mathcal F\big(\phi_0(x),\phi_0(y)\big) \le C\varepsilon L_\varepsilon.
\]
\end{enumerate}
\end{proposition}

\subsection{Almost half-volume superlevels}

In this subsection we show that if $u\in H^{1}(M) \setminus \{0\}$ satisfies $\int_M u =0$, then its superlevels near $\pm 1$ have to be almost half-volume.

\begin{lemma}\label{lem:meanzero_near_half_volume}
Fix $\delta\in(0,1)$ and let $q>1$, $C_1>0$, $\beta>1$ be as in {\rm (vii)} for the potential $W$. There exist constants $C_0=C_0(\delta)>0$ and $C_1'=C_1'(\delta)>0$ such that the following holds. If $u\in H^{1}(M) \setminus \{0\}$ satisfies $\int_M u =0$, then for every $t\in(-1+\delta,\,1-\delta)$,
\[
\Bigl|\operatorname{Vol}\bigl(\{u>t\}\bigr)-\mathfrak h\Bigr|
\ \le\ C_0\,\varepsilon\,E_\varepsilon(u)\ +\ C_1'\,\bigl(\varepsilon E_\varepsilon(u)\bigr)^{1/q}.
\]
\end{lemma}

\begin{proof}
Let
\[
\begin{aligned}
B_\delta(u)
&:=\Bigl\{x\in M:\ 
\min\bigl\{|u(x)-1|,\ |u(x)+1|\bigr\}\ge \delta\Bigr\},\\
A^\pm_\delta(u)
&:=\Bigl\{x\in M:\ |u(x)\mp 1|<\delta\Bigr\}.
\end{aligned}
\]
Then $M=A^+_\delta(u)\sqcup A^-_\delta(u)\sqcup B_\delta(u)$.

\smallskip
\noindent\emph{Claim 1.} There exists $m_\delta=m_\delta(\delta)>0$ such that
$W(s)\ge m_\delta$ whenever $\min\{|s-1|,|s+1|\}\ge\delta$. In particular,
\begin{equation}\label{eq:vol_Bdelta_bound}
\operatorname{Vol}\bigl(B_\delta(u)\bigr)\le \frac{\varepsilon}{m_\delta}\,E_\varepsilon(u).
\end{equation}
Indeed, using {\rm (vii)} we have $W(s)\ge C_1\beta^q$ for $|s|\ge \beta$, while on the compact set
$\{\,|s|\le\beta:\ \min\{|s-1|,|s+1|\}\ge\delta\,\}$ the continuous function $W$ attains a positive minimum
because $W>0$ away from $\{\pm1\}$. Taking $m_\delta$ to be the minimum of these two lower bounds gives the claim,
and \eqref{eq:vol_Bdelta_bound} follows from
\[
E_\varepsilon(u)\ \ge\ \int_{B_\delta(u)}\frac{W(u)}{\varepsilon}\,dV\ \ge\ \frac{m_\delta}{\varepsilon}\,\operatorname{Vol}\bigl(B_\delta(u)\bigr).
\]

\smallskip
Fix $t\in(-1+\delta,1-\delta)$ and set $\Omega_t:=\{u>t\}$.

\smallskip
\noindent\emph{Claim 2.} One has
\begin{equation}\label{eq:Omega_vs_Aplus_clean}
A^+_\delta(u)\subset \Omega_t\subset A^+_\delta(u)\cup B_\delta(u),
\qquad
\Omega_t\cap A^-_\delta(u)=\emptyset,
\end{equation}
and therefore
\begin{equation}\label{eq:Omega_vs_Aplus_vol}
\bigl|\operatorname{Vol}(\Omega_t)-\operatorname{Vol}(A^+_\delta(u))\bigr|\le \operatorname{Vol}\bigl(B_\delta(u)\bigr).
\end{equation}
Indeed, $u\ge 1-\delta>t$ on $A^+_\delta(u)$, so $A^+_\delta(u)\subset \Omega_t$.
Also $u\le -1+\delta<t$ on $A^-_\delta(u)$, so $\Omega_t\cap A^-_\delta(u)=\emptyset$.
Since $M=A^+_\delta(u)\sqcup A^-_\delta(u)\sqcup B_\delta(u)$, this implies
$\Omega_t\subset A^+_\delta(u)\cup B_\delta(u)$, proving \eqref{eq:Omega_vs_Aplus_clean}.
Taking volumes gives \eqref{eq:Omega_vs_Aplus_vol}.

\smallskip
\noindent\emph{Claim 3.} One has
\begin{equation}\label{eq:Aplus_half_clean}
\Bigl|\operatorname{Vol}\bigl(A^+_\delta(u)\bigr)-\mathfrak h\Bigr|
\le \frac12\,\operatorname{Vol}\bigl(B_\delta(u)\bigr)+\frac{1}{2(1-\delta)}\int_{B_\delta(u)}|u|\,dV.
\end{equation}
Indeed, since $\int_M u=0$,
\[
0=\int_{A^+_\delta(u)}u+\int_{A^-_\delta(u)}u+\int_{B_\delta(u)}u.
\]
Using $u\ge 1-\delta$ on $A^+_\delta(u)$ and $u\le -(1-\delta)$ on $A^-_\delta(u)$, we obtain
\begin{equation}
(1-\delta)\,\bigl(\operatorname{Vol}(A^+_\delta(u))-\operatorname{Vol}(A^-_\delta(u))\bigr)
\le \int_{B_\delta(u)}|u|\,dV.
\end{equation}
Note that $\int_M (-u)=0$ and that
$A^+_\delta(-u)=A^-_\delta(u)$, $A^-_\delta(-u)=A^+_\delta(u)$. Applying the same argument to $-u$ yields the reverse inequality
\[
(1-\delta)\,\bigl(\operatorname{Vol}(A^-_\delta(u))-\operatorname{Vol}(A^+_\delta(u))\bigr)
\le \int_{B_\delta(u)}|u|\,dV.
\]
Hence,
\begin{equation}\label{eq:Aplus_minus_Aminus_abs}
\bigl|\operatorname{Vol}(A^+_\delta(u))-\operatorname{Vol}(A^-_\delta(u))\bigr|
\le \frac{1}{1-\delta}\int_{B_\delta(u)}|u|\,dV.
\end{equation}
Since $\operatorname{Vol}(A^+_\delta(u))+\operatorname{Vol}(A^-_\delta(u))=\operatorname{Vol}(M)-\operatorname{Vol}(B_\delta(u))$, we can write
\[
\operatorname{Vol}(A^+_\delta(u))-\mathfrak h
=\frac12\bigl(\operatorname{Vol}(A^+_\delta(u))-\operatorname{Vol}(A^-_\delta(u))\bigr)-\frac12\operatorname{Vol}(B_\delta(u)),
\]
and taking absolute values and using \eqref{eq:Aplus_minus_Aminus_abs} yields \eqref{eq:Aplus_half_clean}.

\smallskip
We now estimate $\int_{B_\delta(u)}|u|$. Splitting into bounded and unbounded parts,
\[
\int_{B_\delta(u)}|u|
\le \int_{B_\delta(u)\cap\{|u|\le \beta\}}|u|+\int_{\{|u|>\beta\}}|u|
\le \beta\,\operatorname{Vol}\bigl(B_\delta(u)\bigr)+\int_{\{|u|>\beta\}}|u|.
\]
By H\"older and {\rm (vii)},
\begin{align*}
\int_{\{|u|>\beta\}} |u|
&\le \operatorname{Vol}(M)^{1-1/q}\Bigl(\int_{\{|u|>\beta\}} |u|^q\Bigr)^{1/q}\\
&\le \operatorname{Vol}(M)^{1-1/q}\Bigl(\frac{1}{C_1}\int_M W(u)\Bigr)^{1/q}\\
&\le \operatorname{Vol}(M)^{1-1/q}\Bigl(\frac{\varepsilon}{C_1}E_\varepsilon(u)\Bigr)^{1/q}.
\end{align*}
Insert this bound into \eqref{eq:Aplus_half_clean} to obtain
\[
\begin{aligned}
\Bigl|\operatorname{Vol}\bigl(A^+_\delta(u)\bigr)-\mathfrak h\Bigr|
&\le \frac12\,\operatorname{Vol}\bigl(B_\delta(u)\bigr)
+\frac{1}{2(1-\delta)}
\Biggl[
\beta\,\operatorname{Vol}\bigl(B_\delta(u)\bigr)\\
&\qquad\qquad
+ \operatorname{Vol}(M)^{1-1/q}\Bigl(\frac{\varepsilon}{C_1}E_\varepsilon(u)\Bigr)^{1/q}
\Biggr].
\end{aligned}
\]
Finally combine this with \eqref{eq:Omega_vs_Aplus_vol} and use the triangle inequality:
\[
\Bigl|\operatorname{Vol}(\Omega_t)-\mathfrak h\Bigr|
\le \bigl|\operatorname{Vol}(\Omega_t)-\operatorname{Vol}(A^+_\delta(u))\bigr|
     +\Bigl|\operatorname{Vol}(A^+_\delta(u))-\mathfrak h\Bigr|
\]
\[
\le \operatorname{Vol}(B_\delta(u))
+\frac12\,\operatorname{Vol}(B_\delta(u))
+\frac{\beta}{2(1-\delta)}\,\operatorname{Vol}(B_\delta(u))
+\frac{\operatorname{Vol}(M)^{1-1/q}}{2(1-\delta)}\Bigl(\frac{\varepsilon}{C_1}E_\varepsilon(u)\Bigr)^{1/q}.
\]

\smallskip
To conclude, use \eqref{eq:vol_Bdelta_bound} to absorb $\operatorname{Vol}(B_\delta(u))$ into $\varepsilon E_\varepsilon(u)$, and set
\[
C_0:=\left(\frac32+\frac{\beta}{2(1-\delta)}\right)\frac{1}{m_\delta},
\qquad
C_1':=\frac{\operatorname{Vol}(M)^{1-1/q}}{2(1-\delta)}\,C_1^{-1/q}.
\]
This gives the stated estimate.
\end{proof}

\subsection{Interpolation and half-volume defect}

In this subsection we carry out the following steps:
\begin{enumerate}
\item[(i)] pass from the flat-fine discrete family given by Proposition \ref{prop:GG_flat_fine_discrete} to a mass-fine discrete family and then to a mass-continuous family in $\mathcal B(M,\mathbb{Z}_2)$;
\item[(ii)] show that this family stays in a small neighborhood of $\mathcal H(M,\mathbb{Z}_2)$, measured by the half-volume defect.
\end{enumerate}

\begin{lemma}\label{lem:dhv_control_by_flat}
There exists a constant $C_I>0$ such that for all $S,T\in \mathcal B(M,\mathbb{Z}_2)$,
\[
|d_{\mathfrak h}(S)-d_{\mathfrak h}(T)|
\le C_I\Bigl(\mathcal F(S,T)+\mathcal F(S,T)^{\frac{n+1}{n}}\Bigr).
\]
\end{lemma}

\begin{proof}
Choose $\Omega_T\in \mathcal C(M)$ with $\partial\Omega_T=T$.
By the definition of the flat norm, for any $\tau>0$ there exist
$R\in \mathbf I_n(M;\mathbb{Z}_2)$ and $Q\in \mathbf I_{n+1}(M;\mathbb{Z}_2)$ such that
\[
S-T = R+\partial Q,
\qquad
\mathbf M(R)+\mathbf M(Q)\le \mathcal F(S,T)+\tau.
\]

Since $Q$ is a top-dimensional integral current mod $2$ on the closed manifold $M$,
there exists a Caccioppoli set $\Omega_Q\in \mathcal C(M)$ such that
$Q=[\![\Omega_Q]\!]$ in $\mathbf I_{n+1}(M;\mathbb{Z}_2)$, hence $\operatorname{Vol}(\Omega_Q)=\mathbf M(Q)$ and $\partial\Omega_Q=\partial Q$.

Moreover, since $S,T\in \mathcal B(M,\mathbb{Z}_2)$ and $\partial Q\in \mathcal B(M,\mathbb{Z}_2)$, we have
\[
R=(S-T)-\partial Q\in \mathcal B(M,\mathbb{Z}_2).
\]
By the isoperimetric inequality \cite{Isoperimetric1, Isoperimetric2, Isoperimetric3}, there exists $U_R\in \mathcal C(M)$ such that $\partial U_R = R$ and
\[
\operatorname{Vol}(U_R)\le C_{\mathrm{iso}}\,\mathbf M(R)^{\frac{n+1}{n}},
\]
where $C_{\mathrm{iso}}>0$.

Define $U:=U_R+\Omega_Q$ (mod 2). Then $U\in\mathcal C(M)$ and
\[
\partial U=\partial U_R+\partial\Omega_Q = R+\partial Q = S-T.
\]
Set $
\Omega_S:=\Omega_T+U$ (mod 2),
so $\Omega_S\in\mathcal C(M)$ and $\partial\Omega_S=T+(S-T)=S$.

By the definition of the half-volume defect and triangle inequality, we obtain
\[
|d_{\mathfrak h}(S)-d_{\mathfrak h}(T)|
= \bigl||\operatorname{Vol}(\Omega_S)-\mathfrak h|-|\operatorname{Vol}(\Omega_T)-\mathfrak h|\bigr|
\le |\operatorname{Vol}(\Omega_S)-\operatorname{Vol}(\Omega_T)|
\le \operatorname{Vol}(U).
\]
Since $\operatorname{Vol}(U)\le \operatorname{Vol}(U_R)+\operatorname{Vol}(\Omega_Q)$, we get
\[
|d_{\mathfrak h}(S)-d_{\mathfrak h}(T)|
\le C_{\mathrm{iso}}\mathbf M(R)^{\frac{n+1}{n}}+\mathbf M(Q)
\le C_{\mathrm{iso}}(\mathcal F(S,T)+\tau)^{\frac{n+1}{n}}+(\mathcal F(S,T)+\tau).
\]
Letting $\tau\downarrow 0$ yields the claim.
\end{proof}

We now pass from the discrete flat-fine family produced by Proposition
\ref{prop:GG_flat_fine_discrete} to a mass-continuous family in $\mathcal B(M,\mathbb{Z}_2)$, and then
show that the half-volume defect remains small.

Fix $\tilde \sigma\in (0,\sigma/2)$, let $X^p$ be a symmetric cubical complex, and let $h:X\to Y$ be continuous and odd. Set
\[
L_\varepsilon:=\sup_{x\in X} E_\varepsilon(h(x)).
\]
Let $\phi_0:X(j)_0\to \mathcal B(M,\mathbb{Z}_2)$ be the even discrete map given by Proposition
\ref{prop:GG_flat_fine_discrete}, and let $\delta\in (0,1)$ be the corresponding parameter.

Define
\[
\eta_\varepsilon:= C_0(\delta)\,\varepsilon L_\varepsilon + C_1'(\delta)\,(\varepsilon L_\varepsilon)^{1/q},
\]
where $C_0(\delta),C_1'(\delta)$ are the constants from Lemma
\ref{lem:meanzero_near_half_volume}. Then for every $x\in X(j)_0$,
\begin{equation}\label{eq:GG_vertex_defect_bound}
d_{\mathfrak h}(\phi_0(x))\le \eta_\varepsilon.
\end{equation}
Indeed, if $\phi_0(x)=\partial[\![\{h_x>s_x\}]\!]$ as in Proposition
\ref{prop:GG_flat_fine_discrete}, then $\int_M h_x=0$ and
$s_x\in (-1+\delta,1-\delta)$, so Lemma \ref{lem:meanzero_near_half_volume} gives
\[
\Bigl|\operatorname{Vol}(\{h_x>s_x\})-\mathfrak h\Bigr|\le \eta_\varepsilon,
\]
and the claim follows from the definition of $d_{\mathfrak h}$.

For convenience, set
\[
\Psi(t):= t+t^{\frac{n+1}{n}}, \qquad t\ge 0.
\]

Recall that the \emph{fineness} of a map $\phi:X(j)_0 \to \mathcal Z(M,\mathbb{Z}_2)$ is defined by
	\[{\bf f}(\phi) = \sup\left\{ {\bf M}(\phi(x)-\phi(y)) : x,y \in X(j)_0 \ \mbox{adjacent vertices} \right\}.\]
This notion can be thought of as the discrete counterpart of the modulus of continuity of a map into $\mathcal Z(M,\mathbb{Z}_2)$ with respect to the mass topology. Similarly, we can consider the \textit{flat fineness} ${\bf f}_{\mathcal{F}}(\phi)$ of a discrete map by replacing the mass with $\mathcal{F}$ in the definition above.

\begin{proposition}\label{prop:interpolation_package_B}
There exist constants $\delta_{\mathcal F,\ast}>0, \delta_\ast>0, C_A>0$
with the following property.

Let $\phi_0:X(j)_0\to \mathcal B(M,\mathbb{Z}_2)$ be an even discrete map satisfying $\mathbf f_{\mathcal F}(\phi_0)\le \delta_{\mathcal F,\ast}$. Then for every $\delta_1\in (0,\delta_\ast)$, after passing to a further subdivision
$X(j+\ell)$, there exist an even discrete map $\phi:X(j+\ell)_0\to \mathcal B(M,\mathbb{Z}_2)$ and an even mass-continuous map $\Phi:X\to \mathcal B(M,\mathbb{Z}_2)$, such that:
\begin{enumerate}
\item[(i)] $\Phi|_{X(j+\ell)_0}=\phi$;
\item[(ii)] $\mathbf f(\phi)\le \delta_1$;
\item[(iii)]
\[
\sup_{x\in X(j+\ell)_0}\mathbf M(\phi(x))
\le \sup_{y\in X(j)_0}\mathbf M(\phi_0(y))+\delta_1;
\]
\item[(iv)] there exists a nearest-vertex projection $\mathbf n:X(j+\ell)_0\to X(j)_0$ such that
\[
\mathcal F\bigl(\phi(x),\phi_0(\mathbf n(x))\bigr)\le \delta_1
\]
for all $x\in X(j+\ell)_0$;
\item[(v)] for every $p$-cell $\alpha$ of $X(j+\ell)$ and every $x,y\in \alpha$,
\[
\mathbf M\bigl(\Phi(x)-\Phi(y)\bigr)\le C_A\,\delta_1.
\]
\end{enumerate}
\end{proposition}

In the application to Proposition \ref{prop:GG_flat_fine_discrete}, the hypothesis
$\mathbf f_{\mathcal F}(\phi_0)\le \delta_{\mathcal F,\ast}$ is ensured by taking $\varepsilon>0$ sufficiently
small, since Proposition \ref{prop:GG_flat_fine_discrete} gives a quantitative bound on the flat fineness of $\phi_0$.

\begin{proof}
This is the standard combination of the flat-to-mass interpolation theorem applied to $\phi_0$ and the Almgren extension theorem (see \cite{MarquesNevesInfinite, MarquesNevesWillmore, Zhou2015}) applied to the resulting mass-fine discrete map.

In the Almgren extension construction, values on a cell are built from the vertex values by adding boundaries of small filling chains. Then since $\phi_0\bigl(X(j)_0\bigr) \subset \mathcal B(M,\mathbb{Z}_2)$ and one only ever adds boundaries in the construction, we have $\Phi(X) \subset \mathcal B(M,\mathbb{Z}_2)$.
\end{proof}

\begin{proposition}\label{prop:defect_after_interpolation}
Assume the notation above, and let $\phi,\Phi$ be given by Proposition
\ref{prop:interpolation_package_B}. Then:
\begin{enumerate}
\item[(i)] for every $x\in X(j+\ell)_0$,
\[
d_{\mathfrak h}(\phi(x))
\le \eta_\varepsilon + C_I\,\Psi(\delta_1);
\]
\item[(ii)] for every $z\in X$,
\[
d_{\mathfrak h}(\Phi(z))
\le \eta_\varepsilon + C_I\,\Psi(\delta_1)+ C_I\,\Psi(C_A\delta_1).
\]
\end{enumerate}
\end{proposition}

\begin{proof}
(i) Fix $x\in X(j+\ell)_0$. By Proposition \ref{prop:interpolation_package_B} (iv) and Lemma \ref{lem:dhv_control_by_flat},
\[
d_{\mathfrak h}(\phi(x))
\le d_{\mathfrak h}\bigl(\phi_0(\mathbf n(x))\bigr)
   + C_I\Bigl(\mathcal F\bigl(\phi(x),\phi_0(\mathbf n(x))\bigr)
   + \mathcal F\bigl(\phi(x),\phi_0(\mathbf n(x))\bigr)^{\frac{n+1}{n}}\Bigr).
\]
Using \eqref{eq:GG_vertex_defect_bound} and
\[
\mathcal F\bigl(\phi(x),\phi_0(\mathbf n(x))\bigr)\le \delta_1,
\]
we obtain
\[
d_{\mathfrak h}(\phi(x))\le \eta_\varepsilon + C_I\,\Psi(\delta_1).
\]
(ii) Fix $z\in X$. Choose a $p$-cell $\alpha$ of $X(j+\ell)$ containing $z$ and a vertex
$x\in \alpha_0$. By Proposition \ref{prop:interpolation_package_B} (v),
\[
\mathbf M\bigl(\Phi(z)-\phi(x)\bigr)\le C_A\,\delta_1.
\]
Since $\mathcal F\le \mathbf M$, Lemma \ref{lem:dhv_control_by_flat} gives
\[
d_{\mathfrak h}(\Phi(z))
\le d_{\mathfrak h}(\phi(x))
   + C_I\Bigl(\mathcal F\bigl(\Phi(z),\phi(x)\bigr)
   + \mathcal F\bigl(\Phi(z),\phi(x)\bigr)^{\frac{n+1}{n}}\Bigr)
\]
\[
\le d_{\mathfrak h}(\phi(x)) + C_I\,\Psi(C_A\delta_1).
\]
Now apply part (i):
\[
d_{\mathfrak h}(\Phi(z))
\le \eta_\varepsilon + C_I\,\Psi(\delta_1)+ C_I\,\Psi(C_A\delta_1).
\]
The proof is complete.
\end{proof}

\subsection{Half-volume correction map}

The next proposition gives a continuous correction to exact half-volume for a compact family of Caccioppoli sets.

For $\Omega:Z\to\mathcal C(M)$ set
\[
T(z):=\partial\Omega(z)\in \mathcal B(M,\mathbb{Z}_2),
\qquad
\delta(z):=\mathfrak h-\operatorname{Vol}(\Omega(z))\in\mathbb{R},
\qquad
\Lambda:=\sup_{z\in Z}\mathbf M(T(z)).
\]

\begin{proposition}
\label{prop:continuous_correction_lifted}
Fix $\Lambda_0<\infty$. Let $Z$ be a compact metric space and let
\[
\Omega:Z\to \mathcal C(M)
\]
be $L^1$-continuous. Assume $\Lambda\le \Lambda_0$. Then there exist constants
\[
\rho_0=\rho_0(\Lambda_0)>0,
\qquad
C=C(\Lambda_0)>0
\]
such that, if
\[
\rho:=\sup_{z\in Z}|\delta(z)|\le \rho_0,
\]
there exists a map
\[
[0,1]\times Z\ni (a,z)\longmapsto \Omega_a(z)\in \mathcal C(M)
\]
with $T_a(z):=\partial\Omega_a(z)$ satisfying:
\begin{enumerate}
\item[(i)] $\Omega_0(z)=\Omega(z)$ and $\operatorname{Vol}(\Omega_1(z))=\mathfrak h$ for all $z\in Z$;
\item[(ii)] $\operatorname{Vol}(\Omega_a(z))=\operatorname{Vol}(\Omega(z))+a\,\delta(z)$ for all $(a,z)\in[0,1]\times Z$;
\item[(iii)] $(a,z)\mapsto \Omega_a(z)$ is $L^1$-continuous and $(a,z)\mapsto T_a(z)$ is flat-continuous;
\item[(iv)] $\mathcal F\bigl(T_a(z),T(z)\bigr)\le C\,\rho$ for all $(a,z)$;
\item[(v)] $\mathbf M\bigl(T_a(z)\bigr)\le \mathbf M\bigl(T(z)\bigr)+C\,\rho$ for all $(a,z)$;
\item[(vi)] if $\delta(z)=0$, then $\Omega_a(z)=\Omega(z)$ for all $a\in[0,1]$.
\end{enumerate}
\end{proposition}

\begin{proof}
\noindent\emph{Step 1: constructing vector fields.}
We will choose $\rho_0\le \mathfrak h/2$, so assume $\rho\le \mathfrak h/2$. Then
\[
\operatorname{Vol}(\Omega(z))=\mathfrak h-\delta(z)\in\Bigl[\frac{\mathfrak h}{2},\frac{3\mathfrak h}{2}\Bigr]
\]
for all $z\in Z$.
Set $\Theta(z):=M\setminus\Omega(z)\in\mathcal C(M)$, so $\Theta:Z\to\mathcal C(M)$ is $L^1$-continuous and
\[
\operatorname{Vol}(\Omega(z))>0,\qquad \operatorname{Vol}(\Theta(z))>0
\]
for all $z\in Z$.

Define the set
\[
\mathcal K:=\Bigl\{\Omega\in\mathcal C(M):\ \mathbf M(\partial\Omega)\le \Lambda_0,\ 
\operatorname{Vol}(\Omega)\in\Bigl[\frac{\mathfrak h}{2},\frac{3\mathfrak h}{2}\Bigr]\Bigr\}.
\]
Then $\Omega(Z)\subset\mathcal K$. By compactness for sets of finite perimeter (see Theorem 12.26 in \cite{Maggi}), $\mathcal K$ is compact in the $L^1$-topology.

\smallskip
\noindent\emph{(a) Local vector fields.}
Fix $\Omega_0\in\mathcal K$ and write $\Theta_0:=M\setminus\Omega_0$.
Choose Lebesgue density points $p_0^-\in \Theta_0, p_0^+\in \Omega_0$. Choose $r_0>0$ so small that the balls $B_0^\pm:=B(p_0^\pm,r_0)$ are embedded and satisfy $\overline{B_0^-}\cap \overline{B_0^+}=\emptyset$.
Set $r_0':=\frac12 r_0$ and $(B_0^\pm)':=B(p_0^\pm,r_0')$.
Define
\[
m_0^0:=\operatorname{Vol}\bigl(\Theta_0\cap (B_0^-)'\bigr)>0,
\qquad
\varepsilon_0^0:=\operatorname{Vol}\bigl(\Theta_0\cap B_0^+\bigr).
\]
Shrinking $r_0$ if necessary, we may assume that
\begin{equation}\label{eq:local_plus_small_balanced}
\varepsilon_0^0\le \frac{1}{32}\,m_0^0\,\frac{\operatorname{Vol}\bigl((B_0^+)'\bigr)}{\operatorname{Vol}(B_0^-)}.
\end{equation}

Choose cutoffs $\phi_0^\pm\in C^\infty_c(B_0^\pm)$ with
\[
0\le \phi_0^\pm\le 1,\qquad \phi_0^\pm\equiv 1 \text{ on } (B_0^\pm)'.
\]
Set
\[
a_0:=\frac{\int_M \phi_0^-\,dV}{\int_M \phi_0^+\,dV},
\qquad
f_0:=-\phi_0^-+a_0\phi_0^+,
\]
so that $\int_M f_0\,dV=0$. Let $u_0\in C^\infty(M)$ solve $\Delta u_0=f_0$ with $\int_M u_0\,dV=0$
(see e.g.\ Theorem 1 in \cite{Donaldson}), and define
\[
X_0:=\nabla u_0,\qquad \operatorname{div} X_0=f_0.
\]

Since $\Omega\mapsto \operatorname{Vol}((M\setminus\Omega)\cap (B_0^-)')$ and $\Omega\mapsto \operatorname{Vol}((M\setminus\Omega)\cap B_0^+)$
are continuous in the $L^1$-topology, there exists an $L^1$-neighborhood $U_0\subset\mathcal K$ of $\Omega_0$ such that,
for every $\Omega\in U_0$ (writing $\Theta:=M\setminus\Omega$),
\begin{equation}\label{eq:local_uniform_bounds_minus_plus}
\operatorname{Vol}\bigl(\Theta\cap (B_0^-)'\bigr)\ge m_0,
\qquad
\operatorname{Vol}\bigl(\Theta\cap B_0^+\bigr)\le \varepsilon_0,
\end{equation}
where we set $m_0:=\frac12\,m_0^0, \varepsilon_0:=2\,\varepsilon_0^0$. Moreover, since $\int_M\phi_0^-\le \operatorname{Vol}(B_0^-)$ and $\int_M\phi_0^+\ge \operatorname{Vol}((B_0^+)')$, we have
\[
a_0\le \frac{\operatorname{Vol}(B_0^-)}{\operatorname{Vol}((B_0^+)')},
\]
and therefore \eqref{eq:local_plus_small_balanced} implies
\begin{equation}\label{eq:local_eps_choice}
a_0\,\varepsilon_0\le \frac{m_0}{8}.
\end{equation}

\smallskip
\noindent\emph{(b) Finite patching over $\mathcal K$.}
The sets $\{U_0\}_{\Omega_0\in\mathcal K}$ form an open cover of the compact space $\mathcal K$.
Choose a finite subcover $U_1,\dots,U_N$, and for each $i$ fix a representative $\Omega_i\in\mathcal K$ and the associated objects
\[
B_i^\pm,\ (B_i^\pm)'\ ,\ \phi_i^\pm,\ a_i,\ f_i,\ X_i,\ m_i,\ \varepsilon_i,
\]
where $m_i,\varepsilon_i$ are the constants produced above for $\Omega_i$, so that the analogues of
\eqref{eq:local_uniform_bounds_minus_plus}--\eqref{eq:local_eps_choice} hold for all $\Omega\in U_i$.

Now define an open cover of $Z$ by $Z_i:=\Omega^{-1}(U_i)\subset Z$.
Choose a continuous partition of unity $\{\chi_i\}_{i=1}^N$ subordinate to $\{Z_i\}$:
\[
0\le \chi_i\le 1,\qquad \operatorname{supp} \chi_i\subset Z_i,\qquad \sum_{i=1}^N \chi_i\equiv 1.
\]
Define a smooth $z$-dependent vector field
\[
X(z):=\sum_{i=1}^N \chi_i(z)\,X_i\in C^\infty(M;TM).
\]
Let $\{\Phi_z^t\}_{|t|<t_\ast}$ be the flow of $X(z)$ (for some $t_\ast>0$ uniform in $z$),
and set
\[
\Theta_t(z):=\Phi_z^t(\Theta(z)),\qquad \Omega_t(z):=M\setminus \Theta_t(z),\qquad
F(z,t):=\operatorname{Vol}(\Omega_t(z))-\operatorname{Vol}(\Omega(z)).
\]

\smallskip
\noindent\emph{(c) Uniform strict monotonicity.}
Since $\Theta(z)$ is Caccioppoli and $X(z)$ is smooth, the transport formula (see e.g. Theorem 15.9 and Proposition 17.8 in \cite{Maggi})
yields that $t\mapsto \operatorname{Vol}(\Theta_t(z))$ (hence $t\mapsto F(z,t)$) is absolutely continuous and for a.e. $t$,
\[
\frac{d}{dt}\operatorname{Vol}(\Theta_t(z))=\int_{\Theta_t(z)} \operatorname{div} X(z)\,dV
=\sum_{i=1}^N \chi_i(z)\int_{\Theta_t(z)} \operatorname{div} X_i\,dV.
\]
For each fixed $i$, using $\operatorname{div} X_i=f_i=-\phi_i^-+a_i\phi_i^+$ and $0\le \phi_i^\pm\le 1$, we obtain for a.e. $t$,
\[
\int_{\Theta_t(z)} \operatorname{div} X_i\,dV
\le -\operatorname{Vol}\bigl(\Theta_t(z)\cap (B_i^-)'\bigr)+a_i\,\operatorname{Vol}\bigl(\Theta_t(z)\cap B_i^+\bigr).
\]
Hence for a.e. $t$,
\begin{equation}\label{eq:global_dvol_ineq}
\frac{d}{dt}\operatorname{Vol}(\Theta_t(z))
\le \sum_{i=1}^N \chi_i(z)\Bigl[-\operatorname{Vol}\bigl(\Theta_t(z)\cap (B_i^-)'\bigr)
+a_i\,\operatorname{Vol}\bigl(\Theta_t(z)\cap B_i^+\bigr)\Bigr].
\end{equation}

Define $m_\ast:=\min_{1\le i\le N} m_i>0$. We claim there exists $\tau\in(0,t_\ast)$, such that for every $i \in \{1, \ldots, N\}$, $z\in\operatorname{supp}\chi_i$ and $t\in[-\tau,\tau]$,
\begin{equation}\label{eq:uniform_time_bounds}
\operatorname{Vol}\bigl(\Theta_t(z)\cap (B_i^-)'\bigr)\ge \frac{m_i}{2},
\qquad
\operatorname{Vol}\bigl(\Theta_t(z)\cap B_i^+\bigr)\le 2\varepsilon_i.
\end{equation}
To prove this, let $\Delta_N:=\{\lambda\in[0,1]^N:\sum_j\lambda_j=1\}$ and for $\lambda\in\Delta_N$ set
\[
X_\lambda:=\sum_{j=1}^N \lambda_j X_j,
\qquad
\Phi_\lambda^t:=\text{flow of }X_\lambda.
\]
For fixed $i$ and fixed ball $B$, the map
\[
(\Omega,\lambda,t)\longmapsto \operatorname{Vol}\bigl(\Phi_\lambda^t(M\setminus\Omega)\cap B\bigr)
\]
is continuous on $\overline U_i\times\Delta_N\times[-t_\ast,t_\ast]$:
$(\lambda,t)\mapsto \Phi_\lambda^t$ depends smoothly on parameters on compact sets,
$\Omega\mapsto M\setminus\Omega$ is $L^1$-continuous, and $A\mapsto \operatorname{Vol}(A\cap B)$ is continuous in $L^1$ for fixed $B$.
At $t=0$, for every $\Omega\in U_i$ the bounds \eqref{eq:local_uniform_bounds_minus_plus} hold with $m_i,\varepsilon_i$.
By uniform continuity on the compact set $\overline U_i\times\Delta_N\times[-t_\ast,t_\ast]$, there exists $\tau_i\in(0,t_\ast)$
such that, for all $\Omega\in \overline U_i$, all $\lambda\in\Delta_N$, and all $t\in[-\tau_i,\tau_i]$,
the bounds \eqref{eq:uniform_time_bounds} hold.
Set $\tau:=\min_{1\le i\le N}\tau_i$.

Now fix $z\in Z$ and $t\in[-\tau,\tau]$. For each $i$ with $\chi_i(z)>0$ we have $z\in Z_i$, hence $\Omega(z)\in U_i$.
Applying \eqref{eq:uniform_time_bounds} and \eqref{eq:local_eps_choice} in \eqref{eq:global_dvol_ineq} yields, for a.e.\ $t\in[-\tau,\tau]$,
\[
\frac{d}{dt}\operatorname{Vol}(\Theta_t(z))
\le \sum_{i=1}^N \chi_i(z)\Bigl(-\frac{m_i}{2}+2a_i\varepsilon_i\Bigr)
\le \sum_{i=1}^N \chi_i(z)\Bigl(-\frac{m_i}{4}\Bigr)
\le -\frac{m_\ast}{4}.
\]
Equivalently,
\[
\frac{d}{dt}F(z,t)=-\frac{d}{dt}\operatorname{Vol}(\Theta_t(z))\ge \frac{m_\ast}{4}
\]
for a.e. $t\in[-\tau,\tau]$ and all $z\in Z$. Integrating and using absolute continuity yields: for all $-\tau\le t_1\le t_2\le \tau$ and $z\in Z$,
\begin{equation}\label{eq:global_strict_mono}
F(z,t_2)-F(z,t_1)\ge \frac{m_\ast}{4}(t_2-t_1).
\end{equation}
In particular, $t\mapsto F(z,t)$ is continuous and strictly increasing on $[-\tau,\tau]$, with $F(z,0)=0$.

\smallskip
\noindent\emph{Step 2: define the corrected family.}
Set $\rho_0:=\min\left\{\frac{\mathfrak h}{2},\,\frac{m_\ast}{4}\tau\right\}$, and assume $\rho\le \rho_0$.
For each $(z,s)\in Z\times[-\rho_0,\rho_0]$, continuity and strict monotonicity give a unique
$t=\theta(z,s)\in[-\tau,\tau]$ such that $F(z,t)=s$.

To see that $\theta$ is continuous, let $(z_k,s_k)\to(z,s)$ and $t_k:=\theta(z_k,s_k)\in[-\tau,\tau]$.
Passing to a subsequence we may assume $t_k\to t$. By continuity of $(z,t)\mapsto F(z,t)$ we have
\[
s=\lim_k s_k=\lim_k F(z_k,t_k)=F(z,t).
\]
Uniqueness from \eqref{eq:global_strict_mono} forces $t=\theta(z,s)$, hence $t_k\to\theta(z,s)$.

Define
\[
t(a,z):=\theta\bigl(z,a\,\delta(z)\bigr),
\qquad
\Omega_a(z):=\Omega_{t(a,z)}(z)=M\setminus \Theta_{t(a,z)}(z).
\]
Then $t(0,z)=0$ and $F(z,t(1,z))=\delta(z)$, so $\Omega_0(z)=\Omega(z)$ and $\operatorname{Vol}(\Omega_1(z))=\mathfrak h.$
Moreover,
\[
\operatorname{Vol}(\Omega_a(z))-\operatorname{Vol}(\Omega(z))=F\bigl(z,t(a,z)\bigr)=a\,\delta(z),
\]
which proves (i)--(ii).

If $\delta(z)=0$, then $t(a,z)=\theta(z,0)=0$ for all $a\in[0,1]$, hence
\[
\Omega_a(z)=\Omega_{t(a,z)}(z)=\Omega_0(z)=\Omega(z),
\]
which proves (vi).

\smallskip
\noindent\emph{Step 3: continuity.}
The map $(a,z)\mapsto a\,\delta(z)$ is continuous and $\theta$ is continuous, hence $(a,z)\mapsto t(a,z)$ is continuous.
Since $z\mapsto \Theta(z)$ is $L^1$-continuous and $(z,t)\mapsto \Phi_z^t$ depends continuously on parameters on compact sets,
the map $(a,z)\mapsto \Theta_{t(a,z)}(z)$, hence $(a,z)\mapsto \Omega_a(z)$, is $L^1$-continuous.
For Caccioppoli sets $A,B$,
\[
\mathcal F(\partial A,\partial B)\le \operatorname{Vol}(A\Delta B),
\]
so $(a,z)\mapsto T_a(z)$ is flat-continuous. This proves (iii).

\smallskip
\noindent\emph{Step 4: mass increase estimate.}
Set
\[
L_0:=\max_{1\le i\le N}\|\nabla X_i\|_{L^\infty(M)}<\infty.
\]
Since $X(z)=\sum_i \chi_i(z)X_i$ with $\chi_i\ge 0$ and $\sum_i\chi_i\equiv 1$, we have
\[
\|\nabla X(z)\|_{L^\infty(M)}\le L_0
\]
for all $z\in Z$. Hence the flow satisfies $\|D\Phi_z^t\|_{L^\infty}\le e^{L_0|t|}$ for $t\in[-\tau,\tau]$.
For any $(n\!-\!1)$-current $S$ and any smooth map $\Phi$, we have
\[
\mathbf M(\Phi_\# S)\le \Bigl(\sup_{x\in M} J_{n-1}\Phi(x)\Bigr)\,\mathbf M(S),
\qquad
J_{n-1}\Phi(x)\le \|D\Phi(x)\|^{n-1}.
\]
Therefore, for the flow,
\[
\mathbf M\bigl((\Phi_z^t)_\#S\bigr)\le e^{(n-1)L_0|t|}\,\mathbf M(S).
\]
Applying this with $S=T(z)$ and $t=t(a,z)$ gives
\[
\mathbf M(T_a(z))\le e^{(n-1)L_0|t(a,z)|}\,\mathbf M(T(z)).
\]
By \eqref{eq:global_strict_mono} with $t_1=0$,
\[
a\,|\delta(z)|=|F(z,t(a,z))|\ge \frac{m_\ast}{4}\,|t(a,z)|,
\]
so $|t(a,z)|\le \frac{4}{m_\ast}\,\rho$. Since $e^{(n-1)L_0|t|}-1\le \frac{e^{(n-1)L_0\tau}-1}{\tau}\,|t|$ for $|t|\le\tau$ and $\mathbf M(T(z))\le \Lambda_0$, we obtain
\[
\mathbf M(T_a(z))-\mathbf M(T(z))
\le \Lambda_0\,\frac{e^{(n-1)L_0\tau}-1}{\tau}\,|t(a,z)|
\le C_1\,\rho,
\]
where $C_1:=\Lambda_0\,\frac{e^{(n-1)L_0\tau}-1}{\tau}\,\frac{4}{m_\ast}$.
This proves (v) with $C\ge C_1$.

\smallskip
\noindent\emph{Step 5: flat distance estimate.}
Set
\[
V_0:=\max_{1\le i\le N}\|X_i\|_{L^\infty(M)}<\infty.
\]
Then $\|X(z)\|_{L^\infty(M)}\le V_0$ for all $z\in Z$.
Fix $(a,z)$ and set $t:=t(a,z)\in[-\tau,\tau]$. Consider the homotopy
\[
h:[0,1]\times M\to M,\qquad h(s,x):=\Phi_z^{s t}(x).
\]
Applying the homotopy flat-norm estimate (see Lemma~4 in \cite{Jerrard}, week 9) to $T(z)$ gives
\[
\mathcal F\bigl((\Phi_z^{t})_\#T(z)-T(z)\bigr)
\le \sup_{(s,x)\in[0,1]\times\operatorname{supp} T(z)}
\Bigl(|\partial_s h(s,x)|\,|\nabla_x h(s,x)|^{n-1}\Bigr)\,\mathbf M\bigl(T(z)\bigr).
\]
We have $|\partial_s h(s,x)|\le |t|\,V_0$ and $|\nabla_x h(s,x)|\le e^{L_0\tau}$ for $s\in[0,1]$, hence
\[
\mathcal F\bigl(T_a(z),T(z)\bigr)
=\mathcal F\bigl((\Phi_z^{t})_\#T(z),T(z)\bigr)
\le |t|\,V_0\,e^{(n-1)L_0\tau}\,\mathbf M\bigl(T(z)\bigr)
\le |t|\,V_0\,e^{(n-1)L_0\tau}\,\Lambda_0.
\]
Using $|t|\le \frac{4}{m_\ast}\rho$ yields
\[
\mathcal F\bigl(T_a(z),T(z)\bigr)\le C_2\,\rho,
\qquad
C_2:=V_0\,e^{(n-1)L_0\tau}\,\Lambda_0\,\frac{4}{m_\ast}.
\]
Taking $C:=\max\{C_1,C_2\}$ gives (iv)--(v).
\end{proof}

\subsection{Proof of Theorem \ref{thm:comp}}

\begin{proof}
Fix $\tilde\sigma\in(0,\sigma/2)$ and choose $\delta_{\mathcal F,\ast},\delta_\ast,C_A$ as in
Proposition \ref{prop:interpolation_package_B}. Let $\delta_0 = \min\{1, \delta_*\}$ and fix $\delta_1\in(0,\delta_0)$.

Choose a sequence $\varepsilon_\nu\downarrow 0$ such that
\[
\tilde c_{\varepsilon_\nu}(p)\to \tilde c(p)=\liminf_{\varepsilon\to 0^+}\tilde c_\varepsilon(p).
\]
Let $\nu$ be sufficiently large and let $X\in\mathcal C_p$ satisfy
\[
\tilde c_{\varepsilon_\nu}(X)\le \tilde c_{\varepsilon_\nu}(p)+\varepsilon_\nu.
\]
Choose $h\in\tilde\Gamma(X)$ such that $h$ is continuous and odd and
\[
L_{\varepsilon_\nu}:=\sup_{x\in X}E_{\varepsilon_\nu}(h(x))
\le \tilde c_{\varepsilon_\nu}(X)+\varepsilon_\nu
\le \tilde c_{\varepsilon_\nu}(p)+2\varepsilon_\nu.
\]

Apply Proposition \ref{prop:GG_flat_fine_discrete} to $h$ to obtain a subdivision level $j$
and an even discrete map
\[
\phi_0:X(j)_0\to \mathcal B(M,\mathbb{Z}_2),\qquad \phi_0(x)=\partial[\![\{h_x>s_x\}]\!],
\]
with the mass bound
\[
\mathbf M(\phi_0(x))\le \frac{E_{\varepsilon_\nu}(h_x)}{4\tilde\sigma}\le \frac{L_{\varepsilon_\nu}}{4\tilde\sigma}.
\]
Moreover, Proposition \ref{prop:GG_flat_fine_discrete} gives a quantitative estimate $\mathbf f_{\mathcal F}(\phi_0)\le C\,\varepsilon_\nu L_{\varepsilon_\nu}$. Hence, by taking $\nu$ sufficiently large (depending on $\delta_1$) we may assume that $\mathbf f_{\mathcal F}(\phi_0)\le \delta_{\mathcal F,\ast}$.

Now apply Proposition \ref{prop:interpolation_package_B} to $\phi_0$ with this $\delta_1$.
We obtain, after passing to a further subdivision $X(j+\ell)$,
an even mass-continuous map $\Phi:X\to \mathcal B(M,\mathbb{Z}_2)$
and an even discrete refinement $\phi:X(j+\ell)_0\to \mathcal B(M,\mathbb{Z}_2)$ such that
\[
\sup_{x\in X}\mathbf M(\Phi(x))
\le \sup_{y\in X(j)_0}\mathbf M(\phi_0(y))+\delta_1
\le \frac{L_{\varepsilon_\nu}}{4\tilde\sigma}+\delta_1
\le \frac{\tilde c_{\varepsilon_\nu}(p)+2\varepsilon_\nu}{4\tilde\sigma}+\delta_1.
\]

Let $\eta_{\varepsilon_\nu}$ be as in \eqref{eq:GG_vertex_defect_bound}.
By Proposition \ref{prop:defect_after_interpolation} we have, for all $z\in X$,
\[
d_{\mathfrak h}(\Phi(z))
\le \eta_{\varepsilon_\nu} + C_I\,\Psi(\delta_1)+ C_I\,\Psi(C_A\delta_1)
=: \rho(\varepsilon_\nu,\delta_1).
\]
By taking $\nu$ sufficiently large (depending on $\delta_1$) we may assume
\[
\rho(\varepsilon_\nu,\delta_1)\le \rho_0(\Lambda_0),
\]
where
\[
\Lambda_0:=\frac{\tilde c(p)+1}{4\tilde\sigma}+1<\infty
\]
and $\rho_0(\Lambda_0)$ is the constant from Proposition
\ref{prop:continuous_correction_lifted}.

Since $\tilde c_{\varepsilon_\nu}(p)\to \tilde c(p)$ and $\varepsilon_\nu\to 0$, for all sufficiently large $\nu$ we have
\[
\frac{\tilde c_{\varepsilon_\nu}(p)+2\varepsilon_\nu}{4\tilde\sigma}+\delta_1
\le \Lambda_0.
\]
Hence
\[
\sup_{x\in X}\mathbf M(\Phi(x))\le \Lambda_0.
\]

Let $\tilde X:=X/\{x\sim \tau(x)\}$ and let $\tilde\Phi:\tilde X\to \mathcal B(M,\mathbb{Z}_2)$ be the induced map.
Form the pullback double cover
\[
\widehat X=\{(\tilde x,\Omega)\in \tilde X\times \mathcal C(M):\ \partial\Omega=\tilde\Phi(\tilde x)\},
\qquad \pi:\widehat X\to \tilde X,
\]
and define the lifted family $\widehat\Omega:\widehat X\to \mathcal C(M)$ by projection.
Then $\widehat\Omega$ is $L^1$-continuous and $\partial\widehat\Omega=\tilde\Phi\circ\pi$. Moreover, for every $\hat x\in\widehat X$ we have
\[
\bigl|\operatorname{Vol}(\widehat\Omega(\hat x))-\mathfrak h\bigr|=d_{\mathfrak h}(\tilde\Phi(\pi(\hat x)))
\le \rho(\varepsilon_\nu,\delta_1),
\]
and
\[
\mathbf M\bigl(\partial\widehat\Omega(\hat x)\bigr)
=\mathbf M\bigl(\tilde\Phi(\pi(\hat x))\bigr)\le \Lambda_0.
\]
Set
\[
D(\hat x):=\mathfrak h-\operatorname{Vol}(\widehat\Omega(\hat x)),
\qquad
\widehat X_-:=\{\hat x\in \widehat X:\ D(\hat x)\ge 0\}.
\]
Since $D$ is continuous and $D(\iota\hat x)=-D(\hat x)$, the set $\widehat X_-$ is compact, and every fiber $\pi^{-1}(\tilde x)$ meets $\widehat X_-$.

Apply Proposition \ref{prop:continuous_correction_lifted} to the restricted family $\widehat\Omega|_{\widehat X_-}:\widehat X_-\to \mathcal C(M)$. We obtain a family
\[
[0,1]\times \widehat X_- \ni (a,\hat x)\longmapsto \widehat\Omega_a^-(\hat x)\in \mathcal C(M)
\]
such that $\operatorname{Vol}(\widehat\Omega_1^-(\hat x))=\mathfrak h$
for all $\hat x\in \widehat X_-$, and $\widehat T_a^-(\hat x):=\partial\widehat\Omega_a^-(\hat x)$ is flat-continuous in $(a,\hat x)$.

We extend this family to all of $\widehat X$ by
\[
\widehat\Omega_a(\hat x):=
\begin{cases}
\widehat\Omega_a^-(\hat x), & \hat x\in \widehat X_-,\\[1ex]
M\setminus \widehat\Omega_a^-(\iota\hat x), & \hat x\notin \widehat X_-.
\end{cases}
\]
If $D(\hat x)=0$, then $\widehat\Omega_a^-(\hat x)=\widehat\Omega(\hat x)$ for all $a\in[0,1]$.
Hence the above definition is $L^1$-continuous across the interface $\{D=0\}$.

Set $\widehat T_a(\hat x):=\partial\widehat\Omega_a(\hat x)$. Then $(a,\hat x)\mapsto \widehat T_a(\hat x)$ is flat-continuous, $
\widehat T_a(\iota\hat x)=\widehat T_a(\hat x)$ for all $(a,\hat x)\in [0,1]\times \widehat X$, $\widehat T_0=\partial\widehat\Omega$ and $\widehat T_1(\hat x)\in \mathcal H(M,\mathbb{Z}_2)$ for all $\hat x\in \widehat X$.
Therefore $\widehat T_a$ descends to a well-defined flat-continuous homotopy
\[
\tilde\Phi_a:\tilde X\to \mathcal B(M,\mathbb{Z}_2),
\qquad
\tilde\Phi_a(\tilde x)=\widehat T_a(\hat x)
\]
for all $\hat x\in\pi^{-1}(\tilde x)$. In particular, $\tilde\Phi_0=\tilde\Phi$ and $\tilde\Phi_1(\tilde X)\subset \mathcal H(M,\mathbb{Z}_2)$.

Define the even map
\[
\Phi_{\mathfrak h}:X\to \mathcal H(M,\mathbb{Z}_2),\qquad \Phi_{\mathfrak h}(x):=\tilde\Phi_1([x]).
\]

We claim that $\Phi_{\mathfrak h}$ satisfies the required mass bound. By part (v) of that proposition, for every $\hat x\in \widehat X_-$,
\[
\mathbf M\bigl(\widehat T_1^-(\hat x)\bigr)
\le \mathbf M\bigl(\partial\widehat\Omega(\hat x)\bigr)+C\,\rho(\varepsilon_\nu,\delta_1).
\]
By the definition of the extension to all of $\widehat X$, the same estimate holds for every
$\hat x\in\widehat X$, hence
\[
\sup_{x\in X}\mathbf M\bigl(\Phi_{\mathfrak h}(x)\bigr)
\le
\sup_{x\in X}\mathbf M\bigl(\Phi(x)\bigr)+C\,\rho(\varepsilon_\nu,\delta_1).
\]

By taking $\nu$ sufficiently large (depending on $\delta_1$), we have
\[
\sup_{x\in X}\mathbf M\bigl(\Phi_{\mathfrak h}(x)\bigr)
\le
\sup_{x\in X}\mathbf M\bigl(\Phi(x)\bigr)+C\,\delta_1.
\]
Using the bound already proved for $\Phi$, we obtain
\[
\sup_{x\in X}\mathbf M\bigl(\Phi_{\mathfrak h}(x)\bigr)
\le
\frac{\tilde c_{\varepsilon_\nu}(p)+2\varepsilon_\nu}{4\tilde\sigma}+ C\,\delta_1.
\]

By construction, $(a,\tilde x)\mapsto \tilde\Phi_a(\tilde x)$ is a flat homotopy in $\mathcal B(M,\mathbb{Z}_2)$
between $\tilde\Phi$ and $\tilde\Phi_1$.

Let $\bar\lambda\in H^1(\mathcal B(M,\mathbb{Z}_2);\mathbb{Z}_2)$ denote the degree-one generator. Since
$\mathcal H(M,\mathbb{Z}_2)\hookrightarrow \mathcal B(M,\mathbb{Z}_2)$ is a homotopy equivalence, the restriction of
$\bar\lambda$ to $\mathcal H(M,\mathbb{Z}_2)$ is the generator $\lambda\in H^1(\mathcal H(M,\mathbb{Z}_2);\mathbb{Z}_2)$.

Since $\tilde\Phi$ is a $p$-sweepout (see Section 6.9 in \cite{GG2}), we have $\tilde\Phi^{\,*}\bar\lambda^p\neq 0$ in $H^p(\tilde X;\mathbb{Z}_2)$. By homotopy invariance, $\tilde\Phi_1^{\,*}\bar\lambda^p=\tilde\Phi^{\,*}\bar\lambda^p\neq 0.$
Because $\tilde\Phi_1(\tilde X)\subset \mathcal H(M,\mathbb{Z}_2)$ and $\bar\lambda|_{\mathcal H(M,\mathbb{Z}_2)}=\lambda$, it follows that $\tilde\Phi_1^{\,*}\lambda^p\neq 0$ in $H^p(\tilde X;\mathbb{Z}_2)$. Therefore $\Phi_\mathfrak h$ is a half-volume $p$-sweepout, and the theorem follows.
\end{proof}

\section{From cycles to $H^1$-functions}
\label{sec:cycles-to-phase}

In this section, we prove the following inequality.

\begin{theorem} \label{thm:lpwp}
For every $p \in \mathbb{N}$, it holds
	\[\frac{1}{2\sigma}\limsup_{\varepsilon \to 0^+} \tilde c_\varepsilon(p) \leq \tilde\omega_p(M).\]
\end{theorem}

Fix $p\in\mathbb{N}$ and $\eta>0$. By definition of the half-volume width $\tilde\omega_p(M)$, there exists a half-volume $p$-sweepout $\Phi:X \to \mathcal H(M;\mathbb{Z}_2)$ with no concentration of mass and satisfying
\[
\sup_{x\in X}\mathbf M(\Phi(x)) \le \tilde\omega_p(M)+\eta.
\]

Recall that
\[
\widehat X
=\{(\tilde x,\Omega)\in \tilde X\times \mathcal C(M):\ \partial\Omega=\tilde \Phi(\tilde x)\}
\]
is the pullback double cover from Subsection \ref{subsec:lifting}, and
\[
\widehat\Omega:\widehat X\to\mathcal C(M),
\qquad
\widehat\Omega(\tilde x,\Omega)=\Omega,
\]
is the associated lifted $L^1$-continuous family of Caccioppoli sets, satisfying
\[
\partial\widehat\Omega=\tilde \Phi\circ\pi,
\qquad
\widehat\Omega(\iota\hat x)=M\setminus \widehat\Omega(\hat x).
\]

\subsection{Almost smooth discretization}\label{subsec:almost-smooth}

In this subsection we record the output of Dey's almost smooth discretization
construction from Section~3.3 of \cite{Dey}. The result produces a discrete family of Caccioppoli sets on a sufficiently fine vertex set whose boundaries are smooth away from a fixed codimension-two set.

Fix $\eta>0$. As in \cite[Section 3.1]{Dey}, choose points $p_1,\dots,p_J\in M$ and radii $r_0>r_1>0, \delta\in(0,r_1)$,
so that the balls $B_i^0:=B(p_i,r_0), B_i^1:=B(p_i,r_1), i\in[J],$
cover $M$, and define the annuli
\[
\begin{aligned}
\mathcal A_0
&:=\bigcup_{i=1}^J A\bigl(p_i,\; r_1-2\delta,\; r_1+\delta\bigr),\\
\mathcal A_1
&:=\bigcup_{i=1}^J A\bigl(p_i,\; r_1-\tfrac32\delta,\; r_1+\tfrac12\delta\bigr),\\
\mathcal A
&:=\bigcup_{i=1}^J A\bigl(p_i,\; r_1-\delta,\; r_1\bigr).
\end{aligned}
\]

Let $\tilde X:=X/\{x\sim \tau(x)\}$ and recall that $\pi:\widehat X\to \tilde X$ is the pullback double cover with deck involution $\iota$. Since $\tilde\Phi:\tilde X\to \mathcal H(M,\mathbb{Z}_2)$ is flat-continuous and $\widehat\Omega:\widehat X\to \mathcal{C}_{\mathfrak h}(M)$ is $L^1$-continuous, by compactness we can choose
$N\in\mathbb{N}$ such that the following hold:
\begin{enumerate}
\item[(i)] if $\tilde x,\tilde x'\in \tilde X(N)_0$ lie in a common cell of $\tilde X(N)$, then
\[
\mathbf M\bigl(\tilde\Phi(\tilde x)-\tilde\Phi(\tilde x')\bigr)<\eta;
\]
\item[(ii)] if $\hat x,\hat x'\in \widehat X(N)_0$ lie in a common cell of $\widehat X(N)$, then
\[
\operatorname{Vol}\bigl(\widehat\Omega(\hat x)\Delta \widehat\Omega(\hat x')\bigr)<\frac{\eta\delta}{J}.
\]
\end{enumerate}

Let $\{c_q:\ q\in[Q]\}$ be the cells of $\tilde X(N)$, indexed so that $\dim(c_{q_1})\le \dim(c_{q_2})$ if $q_1\le q_2$.
For each $q$, write
\[
\pi^{-1}(c_q)=e_q\sqcup f_q,
\qquad
m:=\dim(c_q)=\dim(e_q)=\dim(f_q),
\]
and let $e_q(J)$ and $f_q(J)$ denote the $J$-fold subdivisions of these cells. Thus $e_q(J)_0=e_q\cap \widehat X(NJ)_0$ and $f_q(J)_0=f_q\cap \widehat X(NJ)_0$.

For each $q\in[Q]$ and $i\in[J]$, choose radii $r_i(q)\in(r_1-\delta,r_1)$ as in \cite[(3.14)]{Dey}, and define $B_i(q):=B(p_i,r_i(q))$. Choose $r_\ast$ so that
\[
r_1>r_\ast>\max\{r_i(q):\ i\in[J],\ q\in[Q]\},
\]
and set $B_i:=B(p_i,r_\ast)$ for $i\in[J]$.

We also define the finite family of measurable sets
\[
\mathcal R_1:=\{B_i,\ M\setminus B_i:\ i\in[J]\},
\qquad
\mathcal R_2:=\Bigl\{\bigcap_{j=1}^s U_j:\ s\in\mathbb{N},\ U_j\in\mathcal R_1\Bigr\},
\]
\[
\mathcal R:=\Bigl\{\bigcup_{j=1}^t V_j:\ t\in\mathbb{N},\ V_j\in\mathcal R_2\Bigr\}.
\]

By \cite[Propositions~3.4 and~3.5]{Dey}, after choosing $\gamma=\gamma(\eta)\in\mathbb{N}$ sufficiently large,
there exists a map
\[
\widetilde\Phi_\gamma:\widehat X(N)_0\to \mathcal C(M)
\]
such that, writing
\[
\Sigma_x^\gamma:=\partial \widetilde\Phi_\gamma(x),
\]
the following hold:
\begin{enumerate}
\item[(i)] $\widetilde\Phi_\gamma(\iota x)=M\setminus \widetilde\Phi_\gamma(x)$ for all $x\in\widehat X(N)_0$;
\item[(ii)] each $\Sigma_x^\gamma$ is a smooth, closed, embedded hypersurface;
\item[(iii)]
\[
\operatorname{Vol}\bigl(\widetilde\Phi_\gamma(x)\Delta \widehat\Omega(x)\bigr)<\frac{\eta\delta}{J}
\qquad\forall\,x\in\widehat X(N)_0;
\]
\item[(iv)] for every $x\in\widehat X(N)_0$,
\[
\mathbf M(\Sigma_x^\gamma)\le \mathbf M(\tilde\Phi(\pi(x)))+2\eta;
\]
\item[(v)] for every $x\in\widehat X(N)_0$ and every $i\in[J]$,
\[
\mathbf M(\Sigma_x^\gamma\cap B_i^0)<2\eta,
\qquad
\mathbf M(\Sigma_x^\gamma\cap \mathcal A_0)<2\eta;
\]
\item[(vi)] if $x,x'\in\widehat X(N)_0$ lie in a common cell of $\widehat X(N)$, then for every
$R\in\mathcal R$,
\[
\bigl|\mathbf M(\Sigma_x^\gamma\cap R)-\mathbf M(\Sigma_{x'}^\gamma\cap R)\bigr|<\eta.
\]
\end{enumerate}

\begin{proposition}[Dey, {\cite[Proposition~3.6]{Dey}}]
\label{prop:dey_discretization_package}
There exist finite sets $F_q\subset[Q]$ for $q\in[Q]$, a closed set $S\subset M$, and a map
\[
\widetilde\Psi:\widehat X(NJ)_0\to \mathcal C(M)
\]
such that the following holds. For $v\in\widehat X(NJ)_0$ set $\widetilde\Omega_v:=\widetilde\Psi(v), \widetilde\Omega_{\iota v}:=M\setminus \widetilde\Omega_v$, and $\Sigma_v:=\partial\widetilde\Omega_v$. For each $q\in[Q]$ with $m:=\dim(c_q)$, the restrictions of $\widetilde\Psi$ to
$e_q(J)_0$ and $f_q(J)_0$ satisfy the following properties.

\begin{enumerate}
\item[\rm(P0)$_{m,q}$]
For all $x\in e_q(1)_0\cup f_q(1)_0$,
\[
\widetilde\Psi(x)=\widetilde\Phi_\gamma(x).
\]

\item[\rm(P1)$_{m,q}$]
For all $v\in e_q(J)_0\cup f_q(J)_0$,
\[
M=\widetilde\Omega_v\cup \widetilde\Omega_{\iota v},
\qquad
[\![\widetilde\Omega_v]\!]+[\![\widetilde\Omega_{\iota v}]\!]=[\![M]\!]
\ \text{in }\mathbf I_{n+1}(M;\mathbb{Z}_2),
\]
and $\Sigma_v=\Sigma_{\iota v}$ as sets.

\item[\rm(P2)$_{m,q}$]
For all $v\in e_q(J)_0$,
\[
\widetilde\Omega_v \subset \bigcup_{x\in e_q(1)_0}\widetilde\Phi_\gamma(x),
\]
and for all $v\in f_q(J)_0$,
\[
\widetilde\Omega_v \subset \bigcup_{x\in f_q(1)_0}\widetilde\Phi_\gamma(x).
\]

\item[\rm(P3)$_{m,q}$]
For all $v\in e_q(J)_0$,
\[
\Sigma_v\subset
\left(\bigcup_{x\in e_q(1)_0}\Sigma_x^\gamma\right)
\cup
\left(\bigcup_{i\in[J],\,s\in F_q}\partial B_i(s)\right),
\]
and for all $v\in f_q(J)_0$,
\[
\Sigma_v\subset
\left(\bigcup_{x\in f_q(1)_0}\Sigma_x^\gamma\right)
\cup
\left(\bigcup_{i\in[J],\,s\in F_q}\partial B_i(s)\right).
\]

\item[\rm(P4)$_{m,q}$]
For all $v\in e_q(J)_0\cup f_q(J)_0$, the hypersurface $\Sigma_v\setminus S$ is smooth and embedded.
Moreover, $\Sigma_v\setminus S$ is a disjoint union of open subsets of hypersurfaces belonging to
\[
\{\Sigma_x^\gamma:\ x\in e_q(1)_0\cup f_q(1)_0\}\cup
\{\partial B_i(s):\ i\in[J],\ s\in F_q\}.
\]
In particular, for every $p\in \Sigma_v\setminus S$ there exists a normal geodesic ball $U$
centered at $p$ such that $U\cap \Sigma_v$ is a smooth hypersurface and
\[
U\setminus \Sigma_v=G_1\sqcup G_2,
\]
with $G_1\subset \widetilde\Omega_v$ and $G_2\subset \widetilde\Omega_{\iota v}$.

\item[\rm(P5)$_{m,q}$]
If $v,v'\in e_q(J)_0\cup f_q(J)_0$ are the endpoints of an edge
$e\in e_q(J)_1\cup f_q(J)_1$, then there exists $\ell=\ell(e)\in[J]$ such that
\[
\widetilde\Omega_v\cap\bigl(M\setminus(\mathcal A\cup B_\ell)\bigr)
=
\widetilde\Omega_{v'}\cap\bigl(M\setminus(\mathcal A\cup B_\ell)\bigr),
\]
\[
\widetilde\Omega_{\iota v}\cap\bigl(M\setminus(\mathcal A\cup B_\ell)\bigr)
=
\widetilde\Omega_{\iota v'}\cap\bigl(M\setminus(\mathcal A\cup B_\ell)\bigr),
\]
and hence also
\[
\Sigma_v\cap\bigl(M\setminus(\mathcal A\cup B_\ell)\bigr)
=
\Sigma_{v'}\cap\bigl(M\setminus(\mathcal A\cup B_\ell)\bigr).
\]

\item[\rm(P6)$_{m,q}$]
For all $v\in e_q(J)_0$ and all $l\in[J]$,
\[
\mathbf M\bigl(\Sigma_v\cap B_l^0\bigr)<2^{4m+2}\eta,
\]
and the same estimate holds for $v\in f_q(J)_0$.

\item[\rm(P7)$_{m,q}$]
For all $v\in e_q(J)_0$,
\[
\mathbf M\bigl(\Sigma_v\cap \mathcal A_0\bigr)<2^{4m+2}\eta,
\]
and the same estimate holds for $v\in f_q(J)_0$.

\item[\rm(P8)$_{m,q}$]
For all $v\in e_q(J)_0$, all $x\in e_q(1)_0$, and all $R\in\mathcal R$,
\[
\bigl|\mathbf M(\Sigma_v\cap R)-\mathbf M(\Sigma_x^\gamma\cap R)\bigr|<2^{4m+2}\eta,
\]
and for all $v\in f_q(J)_0$, all $x\in f_q(1)_0$, and all $R\in\mathcal R$,
\[
\bigl|\mathbf M(\Sigma_v\cap R)-\mathbf M(\Sigma_x^\gamma\cap R)\bigr|<2^{4m+2}\eta.
\]
In particular, if
\[
L:=\sup_{x\in X}\mathbf M(\Phi(x)),
\]
then for all $v\in e_q(J)_0\cup f_q(J)_0$,
\[
\mathbf M(\Sigma_v)\le L+\bigl(2^{4m+2}+2\bigr)\eta.
\]
\end{enumerate}
\end{proposition}

\begin{lemma}\label{lem:union_volume_control}
Let $E_1,\dots,E_K\in \mathcal C(M)$ and set $E:=\bigcup_{j=1}^K E_j$.
Then for every $j_0\in\{1,\dots,K\}$,
\[
\bigl|\operatorname{Vol}(E)-\operatorname{Vol}(E_{j_0})\bigr|
\le \sum_{j=1}^K \operatorname{Vol}(E_j\Delta E_{j_0}).
\]
\end{lemma}

\begin{proof}
Since
\[
E=E_{j_0}\cup \bigcup_{j=1}^K(E_j\setminus E_{j_0}),
\]
we have
\[
\operatorname{Vol}(E)-\operatorname{Vol}(E_{j_0})
\le \sum_{j=1}^K \operatorname{Vol}(E_j\setminus E_{j_0})
\le \sum_{j=1}^K \operatorname{Vol}(E_j\Delta E_{j_0}).
\]
The claim follows because $\operatorname{Vol}(E)\ge \operatorname{Vol}(E_{j_0})$.
\end{proof}

Set $V:=\widehat X(NJ)_0$ and $\rho:=\sup_{v\in V}\bigl|\operatorname{Vol}(\widetilde\Omega_v)-\mathfrak h\bigr|$.

\begin{corollary}
\label{cor:dey_small_defect}
Let $\widetilde\Psi:V\to \mathcal C(M)$ be the map given by Proposition
\ref{prop:dey_discretization_package}.
Then
\[
\rho \le \left(3\cdot 2^p+1\right)\frac{\eta\delta}{J}.
\]
\end{corollary}

\begin{proof}
Fix $q\in[Q]$ and let $m:=\dim(c_q)\le p$.
Let $v\in e_q(J)_0$. By \rm(P2)$_{m,q}$,
\[
\widetilde\Omega_v\subset \bigcup_{x\in e_q(1)_0}\widetilde\Phi_\gamma(x).
\]
Choose $x_0\in e_q(1)_0$. For every $x\in e_q(1)_0$, the points $x$ and $x_0$ lie in a common cell
of $\widehat X(N)$, so by the choice of $N$ and $\gamma$,
\[
\begin{aligned}
\operatorname{Vol}\bigl(\widetilde\Phi_\gamma(x)\Delta \widetilde\Phi_\gamma(x_0)\bigr)
&\le
\operatorname{Vol}\bigl(\widetilde\Phi_\gamma(x)\Delta\widehat\Omega(x)\bigr) \\
&\quad
+\operatorname{Vol}\bigl(\widehat\Omega(x)\Delta\widehat\Omega(x_0)\bigr) \\
&\quad
+\operatorname{Vol}\bigl(\widehat\Omega(x_0)\Delta\widetilde\Phi_\gamma(x_0)\bigr) \\
&\le 3\frac{\eta\delta}{J}.
\end{aligned}
\]
Since $|e_q(1)_0|=2^m$, Lemma \ref{lem:union_volume_control} gives
\[
\operatorname{Vol}\Bigl(\bigcup_{x\in e_q(1)_0}\widetilde\Phi_\gamma(x)\Bigr)
\le \operatorname{Vol}(\widetilde\Phi_\gamma(x_0))+3\cdot 2^m\frac{\eta\delta}{J}.
\]
Moreover, since $\widehat\Omega(x_0)\in\mathcal{C}_{\mathfrak h}(M)$,
\[
\bigl|\operatorname{Vol}(\widetilde\Phi_\gamma(x_0))-\mathfrak h\bigr|
\le \operatorname{Vol}\bigl(\widetilde\Phi_\gamma(x_0)\Delta\widehat\Omega(x_0)\bigr)
\le \frac{\eta\delta}{J}.
\]
Hence
\[
\operatorname{Vol}(\widetilde\Omega_v)\le \mathfrak h+\left(3\cdot 2^m+1\right)\frac{\eta\delta}{J}.
\]

Applying the same argument to $\iota v\in f_q(J)_0$ and using \rm(P1)$_{m,q}$, we obtain
\[
\operatorname{Vol}(\widetilde\Omega_v)
=
2\mathfrak h-\operatorname{Vol}(\widetilde\Omega_{\iota v})
\ge
\mathfrak h-\left(3\cdot 2^m+1\right)\frac{\eta\delta}{J}.
\]
Therefore
\[
\bigl|\operatorname{Vol}(\widetilde\Omega_v)-\mathfrak h\bigr|
\le
\left(3\cdot 2^m+1\right)\frac{\eta\delta}{J}
\le
\left(3\cdot 2^p+1\right)\frac{\eta\delta}{J}.
\]
Taking the supremum over $v\in V$ gives the result.
\end{proof}

\subsection{Vertex phase fields and mean defect}

For each vertex $v\in V$, let $d_v:M\to\mathbb{R}$ be the signed distance function associated to $\Sigma_v$, chosen so that $d_v<0$ on $\widetilde\Omega_v$, $d_v>0$ on $M\setminus \widetilde\Omega_v$. By Proposition~\ref{prop:dey_discretization_package} and \cite[Proposition~9.1]{Gua18}, the function $d_v$ is Lipschitz and $|\nabla d_v|=1$ a.e. on $M$.

Let $\mathfrak g:\mathbb{R}\to(-1,1)$ be the unique solution of
\[
\mathfrak g''=W'(\mathfrak g),
\qquad
\mathfrak g(0)=0.
\]
Then $\mathfrak g$ is odd and $\mathfrak g(t)\to \pm 1$ exponentially as $t\to\pm\infty$. For $\varepsilon>0$, set $\mathfrak g_\varepsilon(t):=\mathfrak g(t/\varepsilon)$,
and define $g_\varepsilon:\mathbb{R}\to[-1,1]$ by
\[
g_\varepsilon(t):=
\begin{cases}
\mathfrak g_\varepsilon(t), & |t|\le \sqrt{\varepsilon},\\[0.5ex]
\mathfrak g_\varepsilon(\sqrt{\varepsilon})+\left(\dfrac{t}{\sqrt{\varepsilon}}-1\right)\bigl(1-\mathfrak g_\varepsilon(\sqrt{\varepsilon})\bigr),
& \sqrt{\varepsilon}\le t\le 2\sqrt{\varepsilon},\\[1.5ex]
1, & t\ge 2\sqrt{\varepsilon},\\[0.5ex]
\mathfrak g_\varepsilon(-\sqrt{\varepsilon})+\left(\dfrac{t}{\sqrt{\varepsilon}}+1\right)\bigl(1+\mathfrak g_\varepsilon(-\sqrt{\varepsilon})\bigr),
& -2\sqrt{\varepsilon}\le t\le -\sqrt{\varepsilon},\\[1.5ex]
-1, & t\le -2\sqrt{\varepsilon}.
\end{cases}
\]
Observe that $g_\varepsilon$ is odd and Lipschitz, and $g_\varepsilon(t)=\pm 1$ for $\pm t\ge 2\sqrt{\varepsilon}$.
We define $\vartheta_\varepsilon^v:=g_\varepsilon\circ d_v\in H^1(M)$.

Set
\[
\mathscr S:=
\{\Sigma_x^\gamma:\ x\in \widehat X(N)_0\}
\cup
\{\partial B_i(q):\ i\in[J],\ q\in[Q]\}.
\]

The next lemma is the basic estimate needed in the half-volume setting.

\begin{lemma}\label{lem:vertex_mean_defect}
There exist constants $C=C(\mathscr S)<\infty, \varepsilon_0=\varepsilon_0(\mathscr S)>0,$
such that for every $v\in V$ and every $\varepsilon\in(0,\varepsilon_0)$,
\[
\left|\int_M \vartheta_\varepsilon^v\,dV\right|
\le
2\bigl|\operatorname{Vol}(\widetilde\Omega_v)-\mathfrak h\bigr|
+
C\,\varepsilon.
\]
\end{lemma}

\begin{proof}
Set $s_v:=\chi_{M\setminus \widetilde\Omega_v}-\chi_{\widetilde\Omega_v}$.
Then
\[
\int_M s_v\,dV
=
\operatorname{Vol}(M\setminus \widetilde\Omega_v)-\operatorname{Vol}(\widetilde\Omega_v)
=
2\bigl(\mathfrak h-\operatorname{Vol}(\widetilde\Omega_v)\bigr).
\]
Hence
\[
\left|\int_M \vartheta_\varepsilon^v\,dV\right|
\le
2\bigl|\operatorname{Vol}(\widetilde\Omega_v)-\mathfrak h\bigr|
+
\int_M |\vartheta_\varepsilon^v-s_v|\,dV.
\]
It remains to estimate the second term.

Define
\[
\operatorname{sgn}_0(t):=
\begin{cases}
1,& t>0,\\
0,& t=0,\\
-1,& t<0.
\end{cases}
\]
Since $d_v<0$ on $\widetilde\Omega_v$, $d_v>0$ on $M\setminus \widetilde\Omega_v$, and $\{d_v=0\}=\Sigma_v$ has zero $(n+1)$-dimensional measure, we have $s_v=\operatorname{sgn}_0(d_v)$
a.e. on $M$. Since $g_\varepsilon(t)=\pm 1$ for $\pm t\ge 2\sqrt{\varepsilon}$, it follows that $|\vartheta_\varepsilon^v-s_v|=0$ on $\{|d_v|\ge 2\sqrt{\varepsilon}\}$. Therefore
\[
\int_M |\vartheta_\varepsilon^v-s_v|\,dV
=
\int_{\{|d_v|<2\sqrt{\varepsilon}\}}
|g_\varepsilon(d_v)-\operatorname{sgn}_0(d_v)|\,dV.
\]

By \cite[(3.68)--(3.69)]{Dey}, there exist $ \tau_0=\tau_0(\mathscr S)>0, C_0=C_0(\mathscr S)<\infty$, such that for every $v\in V$ and every $|\tau|\le \tau_0$,
\[
\mathcal H^n(\{d_v=\tau\})\le C_0.
\]
Fix $\varepsilon_0=\varepsilon_0(\mathscr S)>0$ so small that $2\sqrt{\varepsilon_0}\le\tau_0$.
Then for every $\varepsilon\in(0,\varepsilon_0)$, the coarea formula gives
\[
\int_M |\vartheta_\varepsilon^v-s_v|\,dV
=
\int_{-2\sqrt{\varepsilon}}^{2\sqrt{\varepsilon}}
|g_\varepsilon(t)-\operatorname{sgn}_0(t)|\,\mathcal H^n(\{d_v=t\})\,dt.
\]
Since $g_\varepsilon$ is odd, we have
\[
|g_\varepsilon(t)-\operatorname{sgn}_0(t)|=1-g_\varepsilon(t)
\]
for $t\in(0,2\sqrt{\varepsilon})$, and similarly
\[
|g_\varepsilon(-t)-\operatorname{sgn}_0(-t)|=1-g_\varepsilon(t)
\]
for $t\in(0,2\sqrt{\varepsilon})$. Hence
\[
\int_M |\vartheta_\varepsilon^v-s_v|\,dV
\le
2C_0\int_0^{2\sqrt{\varepsilon}} \bigl(1-g_\varepsilon(t)\bigr)\,dt.
\]

We split the last integral into $[0,\sqrt{\varepsilon}]$ and $[\sqrt{\varepsilon},2\sqrt{\varepsilon}]$.
On $[0,\sqrt{\varepsilon}]$, we have $g_\varepsilon(t)=\mathfrak g_\varepsilon(t)=\mathfrak g(t/\varepsilon)$, so
\[
\int_0^{\sqrt{\varepsilon}} \bigl(1-g_\varepsilon(t)\bigr)\,dt
=
\varepsilon\int_0^{1/\sqrt{\varepsilon}} \bigl(1-\mathfrak g(s)\bigr)\,ds
\le
\varepsilon\int_0^\infty \bigl(1-\mathfrak g(s)\bigr)\,ds.
\]
Since $\mathfrak g(s)\to 1$ exponentially as $s\to+\infty$, the last integral is finite, and therefore
\[
\int_0^{\sqrt{\varepsilon}} \bigl(1-g_\varepsilon(t)\bigr)\,dt
\le C\,\varepsilon.
\]

On $[\sqrt{\varepsilon},2\sqrt{\varepsilon}]$, by the definition of $g_\varepsilon$,
\[
0\le 1-g_\varepsilon(t)\le 1-\mathfrak g(1/\sqrt{\varepsilon}),
\]
hence
\[
\int_{\sqrt{\varepsilon}}^{2\sqrt{\varepsilon}} \bigl(1-g_\varepsilon(t)\bigr)\,dt
\le
\sqrt{\varepsilon}\,\bigl(1-\mathfrak g(1/\sqrt{\varepsilon})\bigr).
\]
Again using the exponential convergence of $\mathfrak g$ to $1$, we obtain
\[
\int_{\sqrt{\varepsilon}}^{2\sqrt{\varepsilon}} \bigl(1-g_\varepsilon(t)\bigr)\,dt
\le C\,\varepsilon
\]
for all sufficiently small $\varepsilon$.

Combining the estimates, we conclude that
\[
\int_M |\vartheta_\varepsilon^v-s_v|\,dV\le C(\mathscr S)\,\varepsilon.
\]
This proves the statement.
\end{proof}

\subsection{Gluing construction and mean defect}

We now explain Dey's gluing construction from \cite[Section~3.5]{Dey} in the present setting
and use it to propagate the vertex mean-defect estimate to the final continuous family.

Following \cite[Proposition~3.10]{Dey}, for \(u_0,u_1,w\in H^1(M)\) define
\[
\phi(u_0,u_1,w):=\min\{\max\{u_0,-w\},\,\max\{u_1,w\}\},
\]
\[
\psi(u_0,u_1,w):=\max\{\min\{u_0,w\},\,\min\{u_1,-w\}\},
\]
and
\[
\theta(u_0,u_1,w):=\phi(u_0,u_1,w)^+ + \psi(u_0,u_1,w)^-.
\]
The map \(\theta\) is continuous in \(H^1(M)\), satisfies
\begin{equation}\label{eq:theta_sign_symmetry}
\theta(-u_0,-u_1,w)=-\theta(u_0,u_1,w),
\end{equation}
and, pointwise a.e. on \(M\),
\begin{equation}\label{eq:theta_pointwise_selection}
\theta(y)\in\{u_0(y),u_1(y),w(y),-w(y)\}.
\end{equation}
In particular, if \(|u_0|,|u_1|,|w|\le 1\) a.e., then $|\theta(u_0,u_1,w)|\le 1$ a.e. on $M$.

Let \(f:M\to [\frac13,\frac23]\) be the Morse function from \cite[Section~3.5]{Dey}.
For \(t\in [\frac13,\frac23]\), define the signed distance function
$d^t:M\to \mathbb{R}$ by
\[
d^t(y)=
\begin{cases}
-d(y,f^{-1}(t)),& f(y)\le t,\\[3pt]
\phantom{-}d(y,f^{-1}(t)),& f(y)\ge t.
\end{cases}
\]
Then Dey defines the switch family \(w_\varepsilon:I\to H^1(M)\) by $w_\varepsilon(t)=g_\varepsilon\circ d^t$ for $t\in\Bigl[\frac13,\frac23\Bigr]$, and on the outer intervals by linear interpolation, $w_\varepsilon(t)=1-3t\bigl(1-w_\varepsilon(\tfrac13)\bigr)$ for $0\le t\le \frac13$, and $w_\varepsilon(t)=-1+3(1-t)\bigl(1+w_\varepsilon(\tfrac23)\bigr)$ for $\frac23\le t\le 1$. In particular, \(w_\varepsilon:I\to H^1(M)\) is continuous, $|w_\varepsilon(t)|\le 1$ a.e. on $M$ for every $t\in I$, and $w_\varepsilon(0)=1$, $w_\varepsilon(1)=-1$.

Let \(\alpha\) be a
\(j\)-cell of \(\widehat X(NJ)\). As in Subsection~\ref{subsec:cell_complexes}, let $\Delta_\alpha:I^j\to \alpha$ be the canonical homeomorphism and \(D_\alpha:\alpha\to I^j\) its inverse. Writing
\[
z=(z',z_j)\in I^{j-1}\times I,
\]
set $\widehat\zeta^\alpha_\varepsilon:=\zeta_\varepsilon\circ \Delta_\alpha$. By Proposition~3.11 in \cite{Dey}, there exists a continuous \(\mathbb{Z}_2\)-equivariant map
\[
\zeta_\varepsilon:\widehat X\to H^1(M)\setminus\{0\}
\]
such that:
\begin{enumerate}
\item[(i)] for every vertex \(v\in V\),
\[
\zeta_\varepsilon(v)=\vartheta_\varepsilon^v;
\]
\item[(ii)] for every \(j\)-cell \(\alpha\),
\begin{equation}\label{eq:cellwise_theta_recursion}
\widehat\zeta^\alpha_\varepsilon(z',z_j)
=
\theta\Bigl(
\widehat\zeta^\alpha_\varepsilon(z',0),\,
\widehat\zeta^\alpha_\varepsilon(z',1),\,
w_\varepsilon(z_j)
\Bigr)
\qquad\forall\,(z',z_j)\in I^{j-1}\times I;
\end{equation}
\item[(iii)] for every \(j\)-cell \(\alpha\), there exists a subset \(J_\alpha\subset [J]\) such that $|J_\alpha|\le p,$
and for every \(\hat x\in \alpha\) and every vertex \(v\prec \alpha\),
\begin{equation}\label{eq:gluing_vertex_agreement_outside_bad_set}
\zeta_\varepsilon(\hat x)=\vartheta_\varepsilon^v
\end{equation}
a.e. on $M\setminus\Bigl(\mathcal A_1\cup \bigcup_{i\in J_\alpha} B_i^1\Bigr)$;
\item[(iv)] there exists a constant \(C<\infty\) such that
\begin{equation}\label{eq:glued_family_energy_bound}
\sup_{\hat x\in\widehat X} E_\varepsilon(\zeta_\varepsilon(\hat x))
\le
2\sigma\bigl(\tilde\omega_p(M)+C\eta\bigr)
\end{equation}
for all sufficiently small \(\varepsilon>0\).
\end{enumerate}

\begin{lemma}\label{lem:mean_defect_gluing_step}
Let \(\alpha\) be a cell of \(\widehat X(NJ)\), let \(\hat x\in \alpha\), and let \(v\in \alpha_0\) be a vertex.
Then
\[
\left|
\int_M \zeta_\varepsilon(\hat x)\,dV
-
\int_M \vartheta_\varepsilon^v\,dV
\right|
\le
2\,\operatorname{Vol}\!\Bigl(\mathcal A_1\cup \bigcup_{i\in J_\alpha} B_i^1\Bigr).
\]
\end{lemma}

\begin{proof}
By \eqref{eq:theta_pointwise_selection}, \eqref{eq:cellwise_theta_recursion}, and the facts that
\(|\vartheta_\varepsilon^v|\le 1\) and \(|w_\varepsilon(t)|\le 1\), an induction over the skeleta shows that $|\zeta_\varepsilon(\hat x)|\le 1$ a.e. on $M$ for every \(\hat x\in\widehat X\). Hence
\[
\left|
\int_M \zeta_\varepsilon(\hat x)\,dV
-
\int_M \vartheta_\varepsilon^v\,dV
\right|
\le
\int_M
\bigl|\zeta_\varepsilon(\hat x)-\vartheta_\varepsilon^v\bigr|\,dV.
\]
By \eqref{eq:gluing_vertex_agreement_outside_bad_set}, the integrand vanishes a.e. outside
\[
E_\alpha:=\mathcal A_1\cup \bigcup_{i\in J_\alpha} B_i^1.
\]
Therefore
\[
\int_M
\bigl|\zeta_\varepsilon(\hat x)-\vartheta_\varepsilon^v\bigr|\,dV
=
\int_{E_\alpha}
\bigl|\zeta_\varepsilon(\hat x)-\vartheta_\varepsilon^v\bigr|\,dV
\le
2\,\operatorname{Vol}(E_\alpha),
\]
which proves the lemma.
\end{proof}

\begin{corollary}\label{cor:mean_defect_glued_family}
There exists a constant \(C=C(\mathscr S)<\infty\) such that,
\[
\sup_{\hat x\in \widehat X}
\left|
\int_M \zeta_\varepsilon(\hat x)\,dV
\right|
\le
2\left(3\cdot 2^p+1\right)\frac{\eta\delta}{J}
+
C\,\varepsilon
+
2\,\operatorname{Vol}(\mathcal A_1)
+
2p\max_{i\in[J]}\operatorname{Vol}(B_i^1).
\]
\end{corollary}

\begin{proof}
Fix \(\hat x\in \widehat X\), and let \(\alpha\) be a cell of \(\widehat X(NJ)\) containing \(\hat x\).
Choose a vertex \(v\in \alpha_0\). By Lemma~\ref{lem:mean_defect_gluing_step},
\[
\left|
\int_M \zeta_\varepsilon(\hat x)\,dV
\right|
\le
\left|
\int_M \vartheta_\varepsilon^v\,dV
\right|
+
2\,\operatorname{Vol}\!\Bigl(\mathcal A_1\cup \bigcup_{i\in J_\alpha} B_i^1\Bigr).
\]
Since \(|J_\alpha|\le p\),
\[
\operatorname{Vol}\!\Bigl(\mathcal A_1\cup \bigcup_{i\in J_\alpha} B_i^1\Bigr)
\le
\operatorname{Vol}(\mathcal A_1)+\sum_{i\in J_\alpha}\operatorname{Vol}(B_i^1)
\le
\operatorname{Vol}(\mathcal A_1)+p\max_{i\in[J]}\operatorname{Vol}(B_i^1).
\]
By Corollary~\ref{cor:dey_small_defect} and Lemma~\ref{lem:vertex_mean_defect}, we obtain
\[
\left|
\int_M \zeta_\varepsilon(\hat x)\,dV
\right|
\le
2\left(3\cdot 2^p+1\right)\frac{\eta\delta}{J}
+
C\,\varepsilon
+
2\,\operatorname{Vol}(\mathcal A_1)
+
2p\max_{i\in[J]}\operatorname{Vol}(B_i^1),
\]
as claimed.
\end{proof}

\begin{lemma}\label{lem:vertex_transition_set}
There exist constants \(C<\infty\) and \(\varepsilon_0>0\) such that for every \(v\in V\) and every
\(\varepsilon\in(0,\varepsilon_0)\), the set
\[
K_\varepsilon^v:=\{x\in M:\ |d_v(x)|<2\sqrt\varepsilon\}
\]
satisfies $\vartheta_\varepsilon^v=\pm1$ a.e. on $M\setminus K_\varepsilon^v$, and
\[
\operatorname{Vol}(K_\varepsilon^v)\le C\sqrt\varepsilon.
\]
\end{lemma}

\begin{proof}
Since \(\vartheta_\varepsilon^v=g_\varepsilon\circ d_v\) and \(g_\varepsilon\) satisfies $g_\varepsilon(t)=-1$ for $t\le -2\sqrt\varepsilon$, $g_\varepsilon(t)=1$ for $t\ge 2\sqrt\varepsilon$, we have $\vartheta_\varepsilon^v=\pm1$ a.e. on $M\setminus K_\varepsilon^v$.

It remains to estimate \(\operatorname{Vol}(K_\varepsilon^v)\). Since \(d_v\) is Lipschitz and
\(|\nabla d_v|=1\) a.e.\ on \(M\), the coarea formula gives
\[
\operatorname{Vol}(K_\varepsilon^v)
=
\int_{-2\sqrt\varepsilon}^{2\sqrt\varepsilon}\mathcal H^n(\{d_v=\tau\})\,d\tau.
\]

Now apply \cite[(3.68)--(3.69)]{Dey} with \(U=M\). After shrinking \(\varepsilon_0\) if necessary so that
\(2\sqrt\varepsilon\le \tau_1\) for all \(0<\varepsilon\le \varepsilon_0\), we obtain
\[
\sup_{|\tau|\le 2\sqrt\varepsilon}\mathcal H^n(\{d_v=\tau\})
\le
C\sqrt\varepsilon + (1+C\sqrt\varepsilon)\,\mathcal H^n(\Sigma_v).
\]
By Proposition~\ref{prop:dey_discretization_package}, the hypersurfaces \(\Sigma_v\) have uniformly
bounded area, so after enlarging \(C\) we get
\[
\sup_{v\in V}\sup_{|\tau|\le 2\sqrt\varepsilon}\mathcal H^n(\{d_v=\tau\})\le C
\]
for all $0<\varepsilon\le \varepsilon_0$.
Therefore
\[
\operatorname{Vol}(K_\varepsilon^v)
\le
\int_{-2\sqrt\varepsilon}^{2\sqrt\varepsilon} C\,d\tau
=
4C\sqrt\varepsilon.
\]
Renaming the constant completes the proof.
\end{proof}

\begin{lemma}\label{lem:switch_transition_set}
There exist constants \(C<\infty\) and \(\varepsilon_0>0\) such that for every \(t\in I\) and every
\(\varepsilon\in(0,\varepsilon_0)\), there exists a measurable set \(K_\varepsilon^w(t)\subset M\) satisfying $w_\varepsilon(t)=\pm1$ a.e. on $M\setminus K_\varepsilon^w(t)$, and
\[
\operatorname{Vol}(K_\varepsilon^w(t))\le C\sqrt\varepsilon.
\]
\end{lemma}

\begin{proof}
We split the proof into two cases.

\noindent\emph{Case 1: \(t\in[\frac13,\frac23]\).}
In this case, we have $w_\varepsilon(t)=g_\varepsilon\circ d^t$, where \(d^t\) is the signed distance function to the level set \(f^{-1}(t)\). Define
\[
K_\varepsilon^w(t):=\{x\in M:\ |d^t(x)|<2\sqrt\varepsilon\}.
\]
Since $g_\varepsilon(s)=-1$ for $s\le -2\sqrt\varepsilon$ and $g_\varepsilon(s)=1$ for $s\ge 2\sqrt\varepsilon$, it follows that $w_\varepsilon(t)=\pm1$ a.e. on $M\setminus K_\varepsilon^w(t)$.

It remains to estimate \(\operatorname{Vol}(K_\varepsilon^w(t))\). By \cite[Proposition~9.1]{Gua18},
the signed distance function \(d^t\) is Lipschitz and satisfies $|\nabla d^t|=1$ a.e. on $M$.

Hence the coarea formula gives
\[
\operatorname{Vol}(K_\varepsilon^w(t))
=
\int_{-2\sqrt\varepsilon}^{2\sqrt\varepsilon}\mathcal H^n(\{d^t=\tau\})\,d\tau.
\]

By \cite[Section 9]{Gua18}, there exist constants \(C<\infty\)
and \(\tau_0>0\), independent of \(t\in[\frac13,\frac23]\), such that
\[
\mathcal H^n(\{d^t=\tau\})\le C
\]
for all $t\in\Bigl[\frac13,\frac23\Bigr]$ and $|\tau|\le \tau_0$. After shrinking \(\varepsilon_0>0\) so that \(2\sqrt\varepsilon\le \tau_0\) whenever
\(0<\varepsilon\le \varepsilon_0\), we obtain
\[
\operatorname{Vol}(K_\varepsilon^w(t))
\le
\int_{-2\sqrt\varepsilon}^{2\sqrt\varepsilon} C\,d\tau
\le 4C\sqrt\varepsilon.
\]
After renaming the constant, this proves
\[
\operatorname{Vol}(K_\varepsilon^w(t))\le C\sqrt\varepsilon
\]
for all $t\in\Bigl[\frac13,\frac23\Bigr]$ and $0<\varepsilon\le \varepsilon_0$.

\noindent\emph{Case 2: \(t\in[0,\frac13]\cup[\frac23,1]\).}
Let \(t\in[0,\frac13]\). By construction, \(w_\varepsilon(t)\) is a linear interpolation between the constant function \(1\)
and the endpoint function \(w_\varepsilon(\frac13)\). In particular, if $w_\varepsilon\Bigl(\frac13,x\Bigr)\in\{\pm1\},$ then also $w_\varepsilon(t,x)\in\{\pm1\}.$ Thus the set where \(w_\varepsilon(t)\neq\pm1\) is contained in the set where \(w_\varepsilon(\frac13) \neq \pm 1\). Hence we may take $K_\varepsilon^w(t):=K_\varepsilon^w\Bigl(\frac13\Bigr)$ for $t\in\Bigl[0,\frac13\Bigr]$. By Case 1,
\[
\operatorname{Vol}(K_\varepsilon^w(t))
=
\operatorname{Vol}\!\Bigl(K_\varepsilon^w\Bigl(\frac13\Bigr)\Bigr)
\le C\sqrt\varepsilon.
\]

The same argument with \(w_\varepsilon(\frac23)\) gives the bound on \([\frac23,1]\). This proves the lemma.
\end{proof}

\begin{proposition}\label{prop:zeta_transition_set}
There exist constants \(C<\infty\) and \(\varepsilon_0>0\) such that for every
\(\varepsilon\in(0,\varepsilon_0)\) and every \(\hat x\in\widehat X\), there exists a measurable set
\(K_\varepsilon(\hat x)\subset M\) with $\zeta_\varepsilon(\hat x)=\pm1$ a.e. on $M\setminus K_\varepsilon(\hat x)$, and
\[
\sup_{\hat x\in\widehat X}\operatorname{Vol}(K_\varepsilon(\hat x))\le C\sqrt\varepsilon.
\]
\end{proposition}

\begin{proof}
We argue by induction over the skeleta of \(\widehat X(NJ)\).

If \(\hat x=v\in V=\widehat X(NJ)_0\) is a vertex, set
$K_\varepsilon(v):=K_\varepsilon^v.$ Then the conclusion follows from Lemma~\ref{lem:vertex_transition_set}.

Now let \(\alpha\) be a \(j\)-cell, \(j\ge1\), and write $ \widehat\zeta^\alpha_\varepsilon:=\zeta_\varepsilon\circ\Delta_\alpha.$ For \(z=(z',z_j)\in I^{j-1}\times I\), (\ref{eq:cellwise_theta_recursion}) gives
\[
\widehat\zeta^\alpha_\varepsilon(z',z_j)
=
\theta\Bigl(
\widehat\zeta^\alpha_\varepsilon(z',0),\,
\widehat\zeta^\alpha_\varepsilon(z',1),\,
w_\varepsilon(z_j)
\Bigr).
\]
By the inductive hypothesis, there exist measurable sets $K_\varepsilon^0(z') ,\ K_\varepsilon^1(z') \subset M$ such that $\widehat\zeta^\alpha_\varepsilon(z',0)=\pm1$ on $M\setminus K_\varepsilon^0(z')$ and $\widehat\zeta^\alpha_\varepsilon(z',1)=\pm1$ on $M\setminus K_\varepsilon^1(z')$, with
\[
\operatorname{Vol}(K_\varepsilon^0(z')),\ \operatorname{Vol}(K_\varepsilon^1(z'))\le C_{j-1}\sqrt\varepsilon.
\]
Also, by Lemma~\ref{lem:switch_transition_set}, there exists a measurable set $K_\varepsilon^w(z_j)\subset M$
such that $w_\varepsilon(z_j)=\pm1$ on $M\setminus K_\varepsilon^w(z_j)$, and
\[
\operatorname{Vol}(K_\varepsilon^w(z_j))\le C_w\sqrt\varepsilon.
\]

Set
\[
K_\varepsilon^\alpha(z',z_j)
:=
K_\varepsilon^0(z')\cup K_\varepsilon^1(z')\cup K_\varepsilon^w(z_j).
\]
If \(y\in M\setminus K_\varepsilon^\alpha(z',z_j)\), then
\[
\widehat\zeta^\alpha_\varepsilon(z',0)(y),\
\widehat\zeta^\alpha_\varepsilon(z',1)(y),\
w_\varepsilon(z_j)(y)\in\{\pm1\}.
\]
Since \(\theta(u_0,u_1,w)(y)\in\{u_0(y),u_1(y),w(y),-w(y)\}\), it follows that $\widehat\zeta^\alpha_\varepsilon(z',z_j)(y)\in\{\pm1\}.$
Hence $\widehat\zeta^\alpha_\varepsilon(z',z_j)=\pm1$
a.e. on $M\setminus K_\varepsilon^\alpha(z',z_j)$. Moreover,
\[
\operatorname{Vol}(K_\varepsilon^\alpha(z',z_j))
\le
\operatorname{Vol}(K_\varepsilon^0(z'))+\operatorname{Vol}(K_\varepsilon^1(z'))+\operatorname{Vol}(K_\varepsilon^w(z_j))
\le
(2C_{j-1}+C_w)\sqrt\varepsilon.
\]
Thus the induction closes with $C_j:=2C_{j-1}+C_w$.

Since \(\widehat X(NJ)\) has dimension at most \(p\), after \(p\) steps we obtain a constant
\(C<\infty\) such that for every \(\hat x\in\widehat X\) there exists a measurable set
\(K_\varepsilon(\hat x)\) with $\zeta_\varepsilon(\hat x)=\pm1$
a.e. on $M\setminus K_\varepsilon(\hat x)$, and
\[
\operatorname{Vol}(K_\varepsilon(\hat x))\le C\sqrt\varepsilon.
\]
This proves the proposition.
\end{proof}

\subsection{Mean-zero correction map}

We now record the mean-zero correction map that will be applied to the final glued family.

\begin{proposition}\label{prop:mean_zero_shift}
Let $Z$ be a compact metric space. For each $\varepsilon>0$, let
\[
u_\varepsilon:Z\to H^1(M)
\]
be continuous, and assume:
\begin{enumerate}
\item[(i)] $|u_\varepsilon(z)|\le 1$ a.e.\ on $M$ for every $z\in Z$;
\item[(ii)] there exists a measurable set $K_\varepsilon(z)\subset M$ such that $u_\varepsilon(z)=\pm 1$ a.e. on $M\setminus K_\varepsilon(z)$, and
\[
\sup_{z\in Z}\operatorname{Vol}(K_\varepsilon(z))\le C_0\sqrt{\varepsilon};
\]
\end{enumerate}
Set
\[
m_\varepsilon
:=
\sup_{z\in Z}
\left|
\frac{1}{\operatorname{Vol}(M)}\int_M u_\varepsilon(z)\,dV
\right|,
\qquad
\bar u_\varepsilon(z):=
u_\varepsilon(z)-\frac{1}{\operatorname{Vol}(M)}\int_M u_\varepsilon(z)\,dV.
\]
Then:
\begin{enumerate}
\item[(a)] $z\mapsto \bar u_\varepsilon(z)$ is continuous;
\item[(b)] if $u_\varepsilon$ is odd, then $\bar u_\varepsilon$ is odd;
\item[(c)] $\displaystyle \int_M \bar u_\varepsilon(z)\,dV=0$ for all $z\in Z$;
\item[(d)] there exists $C<\infty$, independent of $\varepsilon$ and $z$, such that
\[
E_\varepsilon(\bar u_\varepsilon(z))
\le
E_\varepsilon(u_\varepsilon(z))
+
C\left(\frac{m_\varepsilon}{\sqrt{\varepsilon}}+\frac{m_\varepsilon^2}{\varepsilon}\right)
\]
for all $z\in Z$.
\end{enumerate}
\end{proposition}

\begin{proof}
Items (a)--(c) are immediate, since $z\mapsto \int_M u_\varepsilon(z)\,dV$
is continuous and odd whenever $u_\varepsilon$ is.

Fix $z\in Z$ and set
\[
c_\varepsilon(z):=\frac{1}{\operatorname{Vol}(M)}\int_M u_\varepsilon(z)\,dV.
\]
Then $|c_\varepsilon(z)|\le m_\varepsilon.$ Since adding a constant does not change the gradient term,
\[
E_\varepsilon(\bar u_\varepsilon(z))-E_\varepsilon(u_\varepsilon(z))
=
\frac1\varepsilon\int_M
\Bigl(W(u_\varepsilon(z)-c_\varepsilon(z))-W(u_\varepsilon(z))\Bigr)\,dV.
\]
We split the integral into $K_\varepsilon(z)$ and its complement.

On $M\setminus K_\varepsilon(z)$ we have $u_\varepsilon(z)=\pm1$ a.e. Since $ W(\pm1)=0$ and $W'(\pm1)=0$, Taylor's theorem gives
\[
|W(\pm1-c)-W(\pm1)|\le C\,c^2
\]
for $|c|\le 1$. Therefore
\[
\frac1\varepsilon\int_{M\setminus K_\varepsilon(z)}
\Bigl|W(u_\varepsilon(z)-c_\varepsilon(z))-W(u_\varepsilon(z))\Bigr|\,dV
\le C\,\frac{m_\varepsilon^2}{\varepsilon}.
\]

On $K_\varepsilon(z)$, since $|u_\varepsilon(z)|\le 1$ and $|c_\varepsilon(z)|\le 1$ for $\varepsilon$ small,
the values of $u_\varepsilon(z)-c_\varepsilon(z)$ remain in a fixed compact interval. Hence $W$ is
Lipschitz there, so
\[
|W(u_\varepsilon(z)-c_\varepsilon(z))-W(u_\varepsilon(z))|
\le C\,|c_\varepsilon(z)|
\le C\,m_\varepsilon.
\]
Using $\operatorname{Vol}(K_\varepsilon(z))\le C_0\sqrt{\varepsilon}$, we obtain
\[
\frac1\varepsilon\int_{K_\varepsilon(z)}
\Bigl|W(u_\varepsilon(z)-c_\varepsilon(z))-W(u_\varepsilon(z))\Bigr|\,dV
\le C\,\frac{m_\varepsilon}{\sqrt{\varepsilon}}.
\]

Combining the two estimates proves (d).
\end{proof}

\subsection{Proof of Theorem \ref{thm:lpwp}}

\begin{proof}
For each $\varepsilon>0$, choose the parameters in Section \ref{subsec:almost-smooth} so that
\[
\frac{\eta\delta}{J}+p\max_{i\in[J]}\operatorname{Vol}(B_i^1)\le \varepsilon,
\qquad
\operatorname{Vol}(\mathcal A_1)\le \varepsilon.
\]
This is possible because, once the finite family of centers $\{p_i\}_{i=1}^J$ and the radius
$r_1$ are fixed, the annuli defining $\mathcal A_1$ are
\[
A\!\left(p_i,\;r_1-\tfrac32\delta,\;r_1+\tfrac12\delta\right),
\]
and hence
\[
\operatorname{Vol}(\mathcal A_1)\le
\sum_{i=1}^J
\operatorname{Vol}\!\left(A\!\left(p_i,\;r_1-\tfrac32\delta,\;r_1+\tfrac12\delta\right)\right)\to 0
\]
as $\delta\to 0$. Then Corollary~\ref{cor:mean_defect_glued_family} yields
\begin{equation}\label{eq:mean_defect_O_eps}
m_\varepsilon
\le C\,\varepsilon.
\end{equation}

Define
\[
\bar\zeta_\varepsilon(\hat x)
:=
\zeta_\varepsilon(\hat x)
-
\frac{1}{\operatorname{Vol}(M)}\int_M \zeta_\varepsilon(\hat x)\,dV.
\]
By Proposition~\ref{prop:mean_zero_shift}, the map $ \bar\zeta_\varepsilon:\widehat X\to Y$ is continuous and odd, and
\begin{equation}\label{eq:energy_after_mean_zero_shift}
\sup_{\hat x\in\widehat X} E_\varepsilon(\bar\zeta_\varepsilon(\hat x))
\le
\sup_{\hat x\in\widehat X} E_\varepsilon(\zeta_\varepsilon(\hat x))
+
C\left(\frac{m_\varepsilon}{\sqrt{\varepsilon}}+\frac{m_\varepsilon^2}{\varepsilon}\right).
\end{equation}
Using \eqref{eq:mean_defect_O_eps}, the last error is $O(\sqrt{\varepsilon})$.

For $\varepsilon>0$ sufficiently small, we also have $\bar\zeta_\varepsilon(\hat x)\neq 0$ for every
$\hat x\in\widehat X$. Indeed, by Proposition~\ref{prop:mean_zero_shift}
and Proposition~\ref{prop:zeta_transition_set}, the set $K_\varepsilon(\hat x)$ satisfies
\[
\operatorname{Vol}(K_\varepsilon(\hat x))\le C\sqrt{\varepsilon},
\]
while one has $\zeta_\varepsilon(\hat x)=\pm1$ a.e. on $M\setminus K_\varepsilon(\hat x)$. Since $m_\varepsilon\to 0$, for small $\varepsilon$ we have $m_\varepsilon\le \frac12$, and therefore $ |\bar\zeta_\varepsilon(\hat x)|\ge \frac12$ a.e. on $M\setminus K_\varepsilon(\hat x)$. Hence $\bar\zeta_\varepsilon(\hat x)$ is not identically zero.

Set $A_\varepsilon:=\bar\zeta_\varepsilon(\widehat X)\subset Y\setminus\{0\}$. Then $A_\varepsilon$ is compact and invariant. Since $ \bar\zeta_\varepsilon:\widehat X\to A_\varepsilon$ is continuous, odd, and surjective, the monotonicity of the $\mathbb{Z}_2$-index gives
\[
\operatorname{ind}_{\mathbb{Z}_2}(A_\varepsilon)\ge \operatorname{ind}_{\mathbb{Z}_2}(\widehat X)\ge p+1.
\]
Therefore $A_\varepsilon\in\mathcal G_p$, and so
\[
\tilde c_\varepsilon(p)
\le
\sup_{u\in A_\varepsilon}E_\varepsilon(u)
=
\sup_{\hat x\in\widehat X}E_\varepsilon(\bar\zeta_\varepsilon(\hat x)).
\]

Finally, using \eqref{eq:glued_family_energy_bound} and
\eqref{eq:energy_after_mean_zero_shift}, we obtain
\[
\tilde c_\varepsilon(p)
\le
2\sigma\tilde\omega_p(M)+C\eta
+
C\left(\frac{m_\varepsilon}{\sqrt{\varepsilon}}+\frac{m_\varepsilon^2}{\varepsilon}\right).
\]
Since $\eta=\eta(\varepsilon)\to 0$, it follows that
\[
\limsup_{\varepsilon\to0^+}\tilde c_\varepsilon(p)\le 2\sigma\,\tilde\omega_p(M).
\]

This completes the proof of the reverse inequality.
\end{proof}


\vspace{0.5cm}
\noindent Department of Mathematics, University of Toronto, Toronto, Canada\\
\textit{E-mail address}: \texttt{talant.talipov@mail.utoronto.ca}

\end{document}